\documentclass{article}

\usepackage{makeidx}
\usepackage{latexsym}
\usepackage{amsfonts}
\usepackage{amssymb}
\usepackage{amsmath}
\usepackage{amstext}
\usepackage{amsthm}
\usepackage{color}
\usepackage{mathrsfs}
\usepackage{relsize}
\usepackage{subfig}
\usepackage{mathtools}

\usepackage{rotating}

\usepackage{easyReview}

\newtheorem{theorem}{Theorem}
\newtheorem{example}[theorem]{Example}
\newtheorem{claim}{Claim}

\usepackage{amsfonts}
\usepackage{amsthm}
\usepackage{amssymb}
\usepackage{tikz}
\usepackage{graphicx}
\usepackage{caption}
\usepackage{subfig}
\usepackage{wasysym}
\usepackage{xcolor}
\usepackage{url}
\usetikzlibrary{calc,intersections}
\newtheorem{lemma}[theorem]{Lemma}
\newtheorem{prop}[theorem]{Proposition}
\newtheorem{corollary}[theorem]{Corollary}
\newtheorem{definition}[theorem]{Definition}

\newtheorem{Prob}[theorem]{Question}

\newtheorem*{main}{Main Theorem}

\usepackage{pgfplots}
\pgfplotsset{compat=1.15}
\usepackage{mathrsfs}
\usetikzlibrary{arrows}
\usepackage{pgfplots}
\pgfplotsset{compat=1.15}
\usepackage{mathrsfs}
\usetikzlibrary{arrows}
\usepackage{algorithm}
\usepackage{algpseudocode}

\usepackage{quoting}
\quotingsetup{font=normalsize}

\usepackage{verbatim}

\newcommand{\cupdot}{\mathbin{\mathaccent\cdot\cup}}

\begin{document}
	
	\definecolor{qqqqff}{rgb}{0.,0.,1.} 
	\definecolor{bfffqq}{rgb}{0.7490196078431373,1.,0.}
	\definecolor{qqwwtt}{rgb}{0.,0.4,0.2}
	\definecolor{zzccqq}{rgb}{0.6,0.8,0.}
	\definecolor{qqffqq}{rgb}{0.,1.,0.}
	\definecolor{ffqqqq}{rgb}{1.,0.,0.}
	\definecolor{xfqqff}{rgb}{0.4980392156862745,0.,1.}
	\definecolor{ffqqff}{rgb}{1.,0.,1.}
	\definecolor{cqcqcq}{rgb}{0.7529411764705882,0.7529411764705882,0.7529411764705882}
	\definecolor{ffffff}{rgb}{1.,1.,1.}
	
	\definecolor{celesteIBM}{RGB}{100,143,255}
	\definecolor{violettoIBM}{RGB}{120,94,240}
	\definecolor{magentaIBM}{RGB}{220,38,127}
	\definecolor{arancioneIBM}{RGB}{254,97,0}
	\definecolor{gialloIBM}{RGB}{255,176,0}
	
	\definecolor{bluTol}{RGB}{51,34,136}
	\definecolor{verdeTol}{RGB}{17,119,51}
	\definecolor{verdeacquaTol}{RGB}{68,170,153}
	\definecolor{celesteTol}{RGB}{136,204,238}
	\definecolor{gialloTol}{RGB}{221,204,119}
	\definecolor{rosaTol}{RGB}{204,102,119}
	\definecolor{violaTol}{RGB}{170,68,153}
	\definecolor{rossoTol}{RGB}{136,34,85}

\title{Critical classes of power graphs and reconstruction of directed power graphs}
\author{\textbf{Daniela Bubboloni} \\
{\small {Dipartimento di Matematica e Informatica ``Ulisse Dini''} }\\
\vspace{-6mm}\\
{\small {Universit\`{a} degli Studi di Firenze} }\\
\vspace{-6mm}\\
{\small {viale Morgagni, 67/a, 50134, Firenze, Italy}}\\
\vspace{-6mm}\\
{\small {e-mail: daniela.bubboloni@unifi.it}}\\
\vspace{-6mm}\\
{\small tel: +39 055 2759667}\\ \and \textbf{Nicolas Pinzauti}
 \\
{\small {Dipartimento di Matematica e Informatica ``Ulisse Dini''} }\\
\vspace{-6mm}\\
{\small {Universit\`{a} degli Studi di Firenze} }\\
\vspace{-6mm}\\
{\small {viale Morgagni, 67/a, 50134, Firenze, Italy}}\\
\vspace{-6mm}\\
{\small {e-mail: nicolas.pinzauti@edu.unifi.it}}\\
\vspace{-6mm}\\
{\small tel: 3345890261}}

\maketitle

\begin{abstract}
In a graph $\Gamma=(V,E)$, we consider the common closed neighbourhood of a subset of vertices and use this notion to introduce a Moore closure operator in $V.$
We also consider the closed twin equivalence relation in  which  two vertices are equivalent if they have the same closed neighbourhood. Those notions are deeply explored when $\Gamma$ is the power graph associated with a finite group $G$. In that case,
 among the corresponding closed twin equivalence classes, we introduce the concepts of plain, compound and critical classes. The study of critical classes, together with properties of the Moore closure operator, allow us to correct a mistake in the proof of {\rm \cite[Theorem 2 ]{Cameron_2}} and to deduce a simple algorithm to reconstruct the directed power graph of a finite group from its undirected counterpart, as asked in \cite[Question 2]{GraphsOnGroups}. 
 \end{abstract}

\vspace{4mm}


\vspace{2mm}


\section{Introduction}

	
	


The interaction between group theory and graph theory has been known since 1878, when Cayley graphs were first defined. In the following years various graphs were associated with groups. In 1955, Brauer and Fowler introduced the commuting graph in  \cite{1955Commuting}, in early 2000 the directed power graphs have been introduced for semigroups in \cite{semigroups,semigroups_2}, and in 2009 Chakrabarty, Ghosh and Sen in \cite{CGS_DefPowerGraph} defined its undirected version, called simply the power graph. Given a group $G$, the directed power graph $\vec{\mathcal{P}}(G)$ has vertex set $G$ and arc set $\{(x,y)\in G^2: x\neq y, \  y=x^m\  \hbox{for some}\  m\in\mathbb{N}\}.$  The power graph $\mathcal{P}(G)$ is its underlying graph. The recent  state of the art for the power graph is well summarized in \cite{Kumar}.

Investigating a group $G$ through a graph $\Gamma(G)$ associated with it allows us to focus on some specific properties of the group, reducing the complexity of the algebraic structure $G$ to the more simple combinatorial object $\Gamma(G)$. Of course, if the reduction of information is dramatic, then the use of  $\Gamma(G)$ for the study of $G$ can be almost useless. 
Among the many graphs associated with groups, the power graph seems to be one of the best to reveal information about the group structure. A result among many justifying this opinion is surely \cite[Theorem 15]{filo}, which shows that if $G$ is a finite simple group then any finite group having the same power graph as $G$ must be isomorphic to $G.$ 
However it is well-known that, given two groups $G_1$ and $G_2$, $\mathcal{P}(G_1)\cong\mathcal{P}(G_2)$ does not imply $G_1\not\cong G_2$. In other words, if a certain graph $\Gamma$ is the power graph of some group, then $\Gamma$ can be the power graph of many non-isomorphic groups. 
For further arguments distinguishing the power graph from other graphs associated with groups,
 the reader is referred to \cite[Introduction]{Bubboloni_3}.

In principle, the directed power graph should encode even more information  than the power graph. Surprisingly this is not the case, at least for finite groups. In this paper we deal only with finite groups, and we call a graph $\Gamma$  a {\it power graph} if there exists at least one finite group $G$ such that $\Gamma=\mathcal{P}(G)$.
As our Main Theorem, we show that, given a power graph $\Gamma=\mathcal{P}(G)$, for a certain finite group $G$, then we can always reconstruct $\vec{\mathcal{P}}(G)$ by purely arithmetical and graph theoretical considerations, without taking into account any group theoretical information about $G$  (see Section \ref{dimmain} for formal details).

\begin{main} We can reconstruct the directed power graph from any power graph.
\end{main}
We also emphasize that the question of the reconstruction of the directed power graph from the power graph makes sense also for infinite groups and there are, in the recent  literature, important contributions for certain classes of infinite groups \cite{zah}. However, in this paper, we deal only with finite groups. 
The proof of the Main Theorem is entirely constructive and gives rise to a precise algorithm whose description and pseudo-code are available in the Appendix.  That completely answers one of the questions recently set by P. J. Cameron \cite[Question  $2$]{GraphsOnGroups}, which asks to find a simple algorithm for constructing the directed power graph from the power graph. 

Evidence of emerging interest in such questions includes a very recent preprint \cite{indiani}, dealing with an algorithm for the reconstruction of the directed power graph from  both the enhanced power graph and the power graph. Recall that the enhanced power graph of a group $G$ has vertex set $G$ and two vertices are adjacent if they generate a cyclic subgroup.
That algorithm is completely independent from ours and based on maximal cyclic subgroups. 

The main scope of our research is theoretical and is mainly inspired by one of the most cited papers about power graphs by  P. J. Cameron \cite{Cameron_2}. We have found there many deep stimulating ideas and tried to exploit them to their maximum extent. 

As a first step we generalize a construction in \cite{Cameron_2}, based on closed neighbourhoods, to any graph giving rise to what we call the {\em neighbourhood closure operator}. Remarkably, this operator turns out to be a Moore closure operator (Section \ref{Moore-sec}). It plays a central role in our paper and, to the best of our knowledge, it does not appear elsewhere in the literature. 
As in \cite{Cameron_2}, we consider the equivalence relation $\mathtt{N}$ which puts in relation two vertices of the power graph of a group $G$ having the same closed neighbourhood.
The study of the partition of the power graph into the corresponding equivalence classes was initiated in \cite{Cameron_2}, but it is far from complete. As in \cite{Cameron_2}, we split the $\mathtt{N}$-classes into two types (which we call {\em plain} and {\em compound}) according to their behaviour with respect to another equivalence relation $\diamond$, for which two vertices are equivalent if they generate the same subgroup of $G$. Then we introduce a further crucial type of $\mathtt{N}$-class, which we call a {\em critical class} (Section \ref{sect-crit-class}). The critical classes arise, in principle, from the necessity to rectify a serious mistake in one argument leading to the main result in \cite{Cameron_2}. 

We emphasize that the error in \cite{Cameron_2} is not rectifiable by modifying some reasoning (see Section \ref{sect-crit-class} for details) and can be fixed only by developing radically new theoretical instruments. An intriguing question is to decide when a critical class is of plain or compound type. That task is completed using a recent result by Feng, Ma and Wang in \cite{Ma et al}. 
We observe that even when all the tools are ready, the proof of the Main Theorem (Section \ref{dimmain}) remains non-obvious and some delicate parts benefit from information coming from general graph theory, such as those describing the influence of the $\mathtt{N}$-classes on the graph automorphisms (Proposition \ref{N-classi_automorphism}).

Finally we stress that the contribution of our paper goes beyond the reconstruction of the directed power graph. Our general analysis of the $\mathtt{N}$-classes can be used for other research on power graphs. Moreover, the creation of a significant Moore closure for graphs seems to be a promising step towards new  general research on graphs. In particular, even though the neighbourhood closure operator is not usually a Kuratowski operator,  the possibility remains open, for certain classes of graphs, of obtaining an interesting topological structure.

\section{Notation and basic facts}
We denote by $\mathbb{N}$ the set of positive integers and we set $\mathbb{N}_0\coloneqq\mathbb{N}\cup\{0\}$. 
For $k\in \mathbb{Z}$ we set $[k]\coloneqq \{x\in \mathbb{N}: x\leq k\}$ and $[k]_0 \coloneqq \{x\in \mathbb{N}_0: x\leq k\}.$ 
Let $X$ be a finite set. We denote by $2^X$ its power set and by $S_{X}$ the symmetric group on $X$. When $X=[k],$ we simplify the notation into $S_k.$ If $\mathcal{K}=\{X_1,\dots,X_r\}$, with $r\in \mathbb{N},$ is a partition of $X$ and $\psi\in S_{X}$, then we denote by $\psi(\mathcal{K})$ the partition of $X$ given by $\{\psi(X_1),\dots, \psi(X_r)\}$.
Given $A,B\subseteq X$, we write $X=A\cupdot B$ if $X=A\cup B$ and $A\cap B=\varnothing.$
Let $\Gamma=(V, E)$ be a graph with vertex set $V\neq \varnothing$ and edge set $E\subseteq\{e\subseteq V: |e|=2\}$. 
For $X\subseteq V$, we denote by $\Gamma_X$ the induced subgraph of $\Gamma$ with vertex set $X$. For $x\in V$, the closed neighbourhood of $x$ in $\Gamma$ is given by  $N[x]=\left\lbrace y \in V | \left\lbrace y, x \right\rbrace \in E\right\rbrace\cup\{x\} $.
Note that, for every $x,y\in V$, we have $y\in N[x]$ if and only if $x\in N[y]$. If $N[x]=V$, then $x$ is called a \emph{star vertex} of $\Gamma$. The set of star vertices of $\Gamma$ is denoted by $\mathcal{S}$. 
For $x,y\in V$, we write $x \mathtt{N}y$ if $N[x]=N[y]$.  The relation $\mathtt{N}$ is an equivalence relation on $V,$ called the \emph{closed twin relation} of $\Gamma.$ Note that $x \mathtt{N}y$, with $x\neq y$, implies $\{x,y\}\in E$.
We denote the $\mathtt{N}$-class of $x\in V$ by $[x]_{\mathtt{N}}$.
 
 \begin{prop}\label{N-classi_automorphism} 
Let $\Gamma=(V,E)$ be a graph and  $\mathcal{K}$ be the partition of $V$ into its $\mathtt{N}$-classes. Then $\times_{_{C\in \mathcal{K}}} S_C\leq Aut(\Gamma).$
	\end{prop}
	\begin{proof} Let $C\in \mathcal{K}$. For $\sigma\in S_C$ define $\tilde{\sigma}: V \rightarrow V$ by  $\tilde{\sigma}(x)=x$ for all $x \in V\setminus C$, and $\tilde{\sigma}(x)=\sigma(x)$ for all $x \in C$. We show that $\tilde{\sigma}$ is a graph automorphism of $\Gamma$. 
	Clearly $\tilde{\sigma}$ is a bijection on $V$. We show that, for every $x,y\in V$ with $x\neq y$, we have $\{x,y\}\in E$ if and only if $\{\tilde{\sigma}(x), \tilde{\sigma}(y)\}\in E$. If both $x,y\in V\setminus C,$ simply observe that $\{\tilde{\sigma}(x), \tilde{\sigma}(y)\}=\{x,y\}.$  If both $x,y\in C,$ we also have $\tilde{\sigma}(x), \tilde{\sigma}(y)\in C$ and thus both $\{x,y\}\in E$ and $\{\tilde{\sigma}(x), \tilde{\sigma}(y)\}\in E$ hold. Consider finally the case $x\in C$ and $y\in V\setminus C$. Then we have $\tilde{\sigma}(x)\in C$, $N[x]=N[\tilde{\sigma}(x)]$ and $\tilde{\sigma}(y)=y$. If $\{x,y\}\in E$, then $y\in N[x]=N[\tilde{\sigma}(x)]$, so that $\{\tilde{\sigma}(x), \tilde{\sigma}(y)\}=\{\tilde{\sigma}(x),y\}\in E.$ Conversely, if $\{\tilde{\sigma}(x), \tilde{\sigma}(y)\}\in E,$ then  we have $\{\tilde{\sigma}(x),y\}\in E$ so that  $y\in N[\tilde{\sigma}(x)]=N[x]$ and hence $\{x,y\}\in E$. Now, using the fact that $\mathcal{K}$ is a partition of $V$,  we immediately get $\times_{_{C\in \mathcal{K}}} S_C\leq Aut(\Gamma)$.
\end{proof}
 
 Throughout the paper, when we desire to emphasize that a set is the vertex set of a graph $\Gamma$ we indicate it by $V_{\Gamma}$. We do the same for the edge set or for other symbols related to the graph $\Gamma$.
 We now briefly recall the power graphs, that is, the graphs on which our paper focuses. 
All the groups considered in this paper are finite.
 Let $G$ be a group with identity element $1$. 
 The \emph{power graph} of $G$, denoted by $\mathcal{P}(G)$, has vertex set $V\coloneqq G$ and $\left\lbrace x,y \right\rbrace \in E$ if $x\neq y$ and there exists a positive integer $m$ such that $x=y^m$ or $y=x^m$. 
The \emph{proper power graph} of $G$, denoted $\mathcal{P}^*(G)$, is the subgraph of $\mathcal{P}(G)$ induced by $G\setminus\{1\}$.
It is easily observed that if $H$ is a subgroup of $G$, then  the subgraph of $\mathcal{P}(G)$ induced by $H$ coincides with $\mathcal{P}(H).$
A graph $\Gamma$ is said to be a power graph if there exists a group $G$ such that $\Gamma=\mathcal{P}(G)$.

We are also interested in the directed version of the power graph. 
The \emph{directed power graph} of $G$, denoted by $\vec{\mathcal{P}}(G)$, has vertex set $V\coloneqq G$ and arc set 
$A\coloneqq\{(x,y)\in G^2: x\neq y, \  y=x^m\  \hbox{for some}\  m\in\mathbb{N}\}.$
Note that $\left\lbrace x,y \right\rbrace \in E$ if and only if at least one of $(x,y)\in A$ and $(y,x)\in A$ holds.
We say that $x,y\in G$, with $x\neq y$, are joined in $\mathcal{P}(G)$ (or in $\vec{\mathcal{P}}(G)$) if 
 $\left\lbrace x,y \right\rbrace \in E$. 
 Let  $X, Y\subseteq G$. We say that $X$ and $Y$ are joined in $\mathcal{P}(G)$ (or in $\vec{\mathcal{P}}(G)$) if there exist $x\in X$ and $y\in Y$ that are joined.
In that case, if it happens that $(x,y)\in A$, then we say that the arc joining $X$ and $Y$ is directed from $X$ to $Y$. 
  
 We denote by $\phi$ the Euler's totient function.  Throughout the paper we freely use the well-known fact that if $m \mid n$, then $\phi(m) \mid \phi(n)$ with equality if and only if $m=n$ or $m$ is odd and $n=2m$.

 \section{A Moore closure operator for graphs }\label{Moore-sec}
  
 We start our research with a section about  closure operators for graphs. The results that we are going to obtain will be used later to deal with the power graphs. However, they seem to be of interest for graphs in general (\cite{Dikran}, \cite{sla}, \cite{sla2}).

 Recall that an operator $c$ on a set $V$ is a function $c:2^V\rightarrow 2^V$, and that $c$ is called a \emph{Moore closure operator} on $V$ (see \cite[Section 4.5.\,a]{librone}) if $c$ is isotone ($A\subseteq B$ 
implies $ c(A)\subseteq c(B)$), extensive ($A\subseteq c(A)$) and idempotent ($c^2(A)=c(A)$);  or a \emph{Kuratowski closure operator} on $V$ (see \cite[Section 5.19]{librone}) if $c$ preserves the empty set ($c(\varnothing)=\varnothing$),  is extensive, idempotent and preserves the binary unions ($c(A\cup B)=c(A)\cup c(B)$). 
It is immediately checked that every Kuratowski closure operator is also a Moore closure operator. Note also that Moore closures always exist, because one can consider the trivial Moore closure given by $c(X)=V$, for all $X\subseteq V.$
If $c$ is a Moore closure operator, then one defines a subset $X$ of $V$ to be closed if $c(X)=X.$ The Moore closure operators are interesting in pure and applied mathematics and computer science as very well described in \cite{mo}. Non-trivial examples  naturally arise in algebra. For instance, if $G$ is a group and $c:2^G\rightarrow 2^G$ is defined, for every $X\subseteq G$, by $c(X)\coloneqq\langle X\rangle$ then $c$ is a Moore closure operator. 
Additionally, the Kuratowski closure operators are of huge importance, because the closed sets of such operators give a topological structure to  $V$.
  
In \cite{Cameron_2}, for a subset $X$ of vertices in a power graph, the sets $N[X]=\bigcap_{x\in X}  N[x]$ and $\hat{X}=N[N[X]]$ are defined and 
fruitfully used when $X$ is an $\mathtt{N}$-class. They are also considered in \cite{zah}, dealing with some subsets in power graphs of infinite groups. 
 It seems important to recognize the role of operators of such objects, generalize them to any graph and explore their properties. 
  
 Let $\Gamma=(V, E)$ be a graph. For  $X\subseteq V$ we define  the  \emph{common closed neighbourhood} of $X$ by
	$$N[X]\coloneqq\begin{cases}
\bigcap_{x\in X}  N[x]&{\text{\rm if }} X\neq \varnothing \\
&\\
V&{\text{\rm if }} X= \varnothing.
\end{cases}$$
For instance, one has $N[V]=\mathcal{S}$.
The consideration of the common closed neighbourhoods in graph theory dates back at least to Lov\'asz, in 1978, with the famous construction of the neighbour complex $\mathcal{N}(\Gamma)$, which allowed him to prove the Kneser's conjecture \cite{lov}. Recall that $\mathcal{N}(\Gamma)$ has vertex set $V$ and simplices given by those $X\subsetneq V$ having $N[X]\neq\varnothing$.
We warn the reader that 
typically, in the literature, the symbol $N[X]$ stands for the union of the closed neighbourhoods of the vertices in $X$ (see, for instance, \cite[Section 2.1]{dom}), and not for their intersection. 

We next define, for every $X\subseteq V$,  the \emph{neighbourhood closure} of $X$ by $\hat{X}\coloneqq N[N[X]]$. That gives rise to the \emph{neighbourhood closure operator}  on $V$, which plays a main role in our research.
Note that  $\hat{\varnothing}=\mathcal{S}$, $\hat{V}=V$ and $\hat{\mathcal{S}}=\mathcal{S}.$  

We present now some useful properties of the common closed neighbourhoods and of the neighbourhood closure  operator. 
\begin{prop}\label{operatoreChiusura} 
	Let $\Gamma=(V,E)$ be a graph and  $\mathtt{N}$ be its closed twin relation.
	Then, for every $X\subseteq V$, the followings facts hold:
		\begin{enumerate}
		\item[$(i)$]  $ \hat{X}\supseteq X\cup \mathcal{S}$ and if $\hat{X}\neq \varnothing$, then $\hat{X}$ is a union of $\mathtt{N}$-classes; 

			\item[$(ii)$] If $A \subseteq B\subseteq V$ then $N[A] \supseteq N[B]$ and $\hat{A}\subseteq \hat{B}$;
			
			\item[$(iii)$] $N[X]=N[\hat{X}]$ and $\widehat{( \hat{X})}= \hat{X}$;
			\item[$(iv)$]  If $A, B\subseteq V$ then $N[A\cup B]=N[A]\cap N[B]$ and $\widehat{A\cup B}\supseteq \hat{A}\cup \hat{B}$;
			\item[$(v)$] If $C$ is an $\mathtt{N}$-class and $y\in C$, then  $N[C]=N[y]$ and $\hat{C}=\bigcap_{z \in N[y]} N[z] \subseteq N[y] $.

		\end{enumerate}
		\end{prop}
	\begin{proof}
	$(i)$ Let $X\subseteq V$.  Assume first that $N[X]= \varnothing$. Then $$\hat{X}=N[\varnothing]=V\supseteq X \cup \mathcal{S}.$$ Assume next that $N[X]\neq\varnothing$.  Then $ \hat{X}=\bigcap_{u\in N[X]}  N[u]$. Pick $x\in X$ and let $u\in N[X]$. Then we have $u\in N[x]$ and thus also $x\in N[u].$ It follows that $x\in \hat{X}$ and hence $ \hat{X}\supseteq X$. We finally note that
$$\hat{X}=\bigcap_{x\in N[X]}  N[x] \supseteq\bigcap_{x\in V}  N[x]=\mathcal{S}.$$

We now show that, when $\hat{X}\not=\varnothing$, $\hat{X}$ is a union of $\mathtt{N}$-classes.  Let $x\in \hat{X}$. We want to show that if $v\in V$ is such that $N[x]=N[v]$, then $v\in  \hat{X}$. By definition of $\hat{X}$, we have that $x\in N[z]$, for all $z\in N[X]$. Thus, for every  $z\in N[X]$, we have  $z\in N[x]=N[v]$ and hence $v\in N[z]$, which gives $v\in  \hat{X}$.
\smallskip

	$(ii)$ Let $A \subseteq B\subseteq V$. We first show that $N[A] \supseteq N[B]$. If $A=\varnothing$, then $N[A]=V \supseteq N[B]$ holds. If $A\neq \varnothing$, then also $B \neq \varnothing,$ so that $N[A] =\bigcap_{x\in A}  N[x]\supseteq \bigcap_{x\in B}  N[x]=N[B].$
 Applying the established inequality to the subsets $N[A] \supseteq N[B]$ we now obtain $\hat{A}=N[N[A]] \subseteq N[N[B]]=\hat{B}.$
	
\smallskip

$(iii)$ We show first that $N[X]=N[\hat{X}]$. By $(i)$, we have $\hat{X}\supseteq X$ and hence, by $(ii)$, we get $N[X]\supseteq N[\hat{X}]$. Assume, by contradiction, that $N[X]\supsetneq N[\hat{X}]$. Then there exists $u\in N[X]\setminus N[\hat{X}]$. In particular, $N[X]\neq \varnothing$ and $N[\hat{X}]\neq V.$ Thus $N[\hat{X}]= \bigcap_{x\in \hat{X}}  N[x]$ and, by the fact that $u\notin N[\hat{X}]$,we deduce that there exists $\hat{x}\in \hat{X}$ such that $u\notin N[\hat{x}]$. On the other hand, since $N[X]\neq \varnothing$, we have that $\hat{X}= \bigcap_{x\in N[X]}  N[x]$. As a consequence, by $u\in N[X]$, we obtain $\hat{X}\subseteq N[u].$ Thus $\hat{x}\in N[u] $, that is, $u\in N[\hat{x}]$, a contradiction.

Now, we immediately  deduce also that  $\widehat{( \hat{X})}=N[N[\hat{X}]]=N[N[ X]]= \hat{X}.$
\smallskip

$(iv)$  Let $A, B\subseteq V$.  If one of $A$ and $B$ is empty, then the equality $N[A\cup B]=N[A]\cap N[B]$ trivially follows. Assume that they are both nonempty. Then also $A\cup B$ is nonempty and we have $$N[A\cup B]= \bigcap_{x\in A\cup B}  N[x]=\bigcap_{x\in A}  N[x]\cap \bigcap_{x\in B}  N[x]=N[A]\cap N[B].$$
As a consequence, by $(ii)$, we have
$$\widehat{A\cup B}=N[N[A\cup B]]=N[N[A]\cap N[B]]\supseteq N[N[A]]\cup N[N[B]]=\hat{A}\cup \hat{B}.$$

$(v)$ Let $y\in C.$ Then $C=[y]_\mathtt{N}\neq \varnothing$ so that, by the definition of the relation $\mathtt{N}$, we deduce
	$$N[C]=\bigcap_{x\in [y]_\mathtt{N}}  N[x]= N[y]\neq \varnothing.$$
	Since $y\in N[y]$, it  follows that $\hat{C}=N[N[y]]=\bigcap_{z \in N[y]} N[z] \subseteq N[y].$
	\end{proof}
Observe that, as an easy consequence of Proposition \ref{operatoreChiusura}, the neighbourhood  closure of a set of vertices is empty if and only if that set is empty and the graph admits no star. We also observe that the inclusion $\widehat{A\cup B}\supseteq \hat{A}\cup \hat{B}$ in Proposition \ref{operatoreChiusura}\,$(iv)$ is generally proper as the following example shows.
\begin{example}\label{ex1} {\rm Let $\Gamma=(V,E)$, where $V=[5]$ and $E=\{\{1,2\}, \{2,3\}, \{3,1\},\{3,4\}\}$. If $A\coloneqq\{1\}$ and $B\coloneqq\{5\}$, then it is immediately checked that $\widehat{A\cup B}=V$ and $\hat{A}\cup \hat{B}=V\setminus\{4\}\subsetneq \widehat{A\cup B}.$}
\end{example}

\noindent The next result is now an immediate consequence of Proposition \ref{operatoreChiusura}\,$(i)$-$(iii)$ and of Example \ref{ex1}.
\begin{corollary}\label{moore}  The neighbourhood closure operator is a Moore closure operator which is not, in general, a Kuratowski closure operator.
\end{corollary}
Clearly there exist graphs for which the neighbourhood closure operator is a Kuratowski closure operator. For instance, this is the case for the totally disconnected graph on a set $V$. We set then the following problem. 
\begin{Prob}\label{prob1}
For which families of graphs is the neighbourhood closure operator also a Kuratowski closure operator?
\end{Prob}
Note that,  in order to respect the preservation of the empty set, we need $\mathcal{S}=\varnothing$. But this is not enough to guarantee that our operator is a Kuratowski closure operator. Indeed, as the Example \ref{ex1} shows, even when the star set is empty the preservation of binary unions is not necessarily satisfied.
A deep enquiry on those aspects is out of the scope of the present paper, but surely deserves future attention.

\section{Preliminary results on power graphs}	
	
We now come back to the focus of our research, that is, to power graphs. Let $G$ be a group and $x,y\in G$. We write $x\diamond y$ if $\langle x \rangle = \langle y \rangle$. Note that $x\diamond y$ if  and only if $x=y$ or both $(x,y)$ and $(y,x)$ are arcs of $\vec{\mathcal{P}}(G)$. 
 The relation $\diamond$ is clearly an equivalence relation in $G$. We denote the $\diamond$-class of  $x\in G$ by $[x]_{\diamond}$. Note that $[x]_{\diamond}$ is the set of generators of $\langle x \rangle$. In particular, $|[x]_{\diamond}|=\phi(o(x))$.
 It is easily checked that, if $\mathtt{N}$ is the closed twin relation of $\mathcal{P}(G)$, then
  $x\diamond y$ implies $x \mathtt{N}y$. In other words the relation $\diamond$ is a refinement of the relation $\mathtt{N}$. As a consequence an $\mathtt{N}$-class is a union of $\diamond$-classes. Moreover, it is immediately observed that, for every $x\in G$, we have 
$[x]_{\diamond}\subseteq [x]_{\mathtt{N}}\subseteq N_{\mathcal{P}(G)}[x]$ and $\langle x \rangle\subseteq  N_{\mathcal{P}(G)}[x].$
In particular, two distinct elements of $G$ in the same $\mathtt{N}$-class are joined  and $\mathcal{P}(G)_{[x]_{\mathtt{N}}}$ is a complete graph. Note that $[1]_{\mathtt{N}}=\mathcal{S}_{\mathcal{P}(G)}$. For this reason, we call the $\mathtt{N}$-class of the identity element of $G$ the \emph{star class}.
	When the star class is not equal to $\{1\}$, there are well-known consequences for the structure of the group and the nature of the star class itself.
	\begin{prop}{\em \cite[Proposition 4]{Cameron_2}}
		\label{S>1}
		Let $G$ be a group with $|G|=n$ such that $|\mathcal{S}_{\mathcal{P}(G)}|>1$. Then one of the following facts holds:
		\begin{itemize}
			\item[\rm (a)] $G$ is a cyclic $p$-group and $\mathcal{S}_{\mathcal{P}(G)}=G$;
			\item[\rm (b)] $G$ is cyclic, not a $p$-group, and $\mathcal{S}_{\mathcal{P}(G)}$ consists of $1$ and  of the generators of $G$, so that $|\mathcal{S}_{\mathcal{P}(G)}|=1 + \phi(n)$;
			\item[\rm (c)] $G$ is a generalized quaternion $2$-group and $\mathcal{S}_{\mathcal{P}(G)}$ contains $1$ and the unique involution of $G$, so that $|\mathcal{S}_{\mathcal{P}(G)}|=2$.
		\end{itemize}
In particular, $\mathcal{P}(G)$ is a complete graph if and only if $G\cong C_{p^m}$ for some prime number $p$ and some $m\in \mathbb{N}_0.$ 
	\end{prop}

We now present a useful auxiliary result. For simplicity, from now on, when a single group $G$ is under consideration, we completely omit all the subscripts $\mathcal{P}(G)$ in our notation and assume that $\mathtt{N}$ is the closed twin relation of $\mathcal{P}(G)$.
\begin{lemma}\label{NdipotenzeDip} 
Let $G$ be a group and $x,y\in G$ be elements having orders that are powers of the same prime number $p$, and such that $o(x)\leq o(y)$. Then $N[x]\supseteq N[y]$ if and only if $x$ is a power of $y$.
\end{lemma}

\begin{proof}  Assume that $N[x]\supseteq N[y]$. Then $y\in N[x]$, which implies $\langle x\rangle \leq \langle y\rangle$ or $\langle y\rangle \leq \langle x\rangle$. Since $o(x)\leq o(y)$, we then have $\langle x\rangle \leq \langle y\rangle$ so that $x$ is a power of $y$.
Conversely,  assume that  $x$ is a power of $y$. Pick $z\in N[y]$. If $y$ is a power of $z$, then $x$ is a power of $z$ too. If $z$ is a power of $y$, then $x$ and $z$  belong to the cyclic $p$-group generated by $y$. Hence, since $\mathcal{P}(\langle y \rangle)$ is complete, we have $z\in N[x].$
\end{proof}
	
	\section{$\mathtt{N}$-classes in  power graphs}

	We have observed that an $\mathtt{N}$-class in a power graph is a union of $\diamond$-classes. Hence we can sensibly distinguish two types of $\mathtt{N}$-classes.
	\begin{definition}
		\label{NplainType}\rm Let $G$ be a group and $C$ be an $\mathtt{N}$-class. We say that $C$ is a \emph{class of plain type} if 
		$C$ is a single  $\diamond$-class; $C$ is a \emph{class of compound type} if $C$ is the union of at least two $\diamond$-classes.
	\end{definition}

Those classes are called, in \cite{Cameron_2}, class of type $(a)$ and $(b)$ respectively.
The classes of plain type are easily characterized by the order of their elements.

\begin{lemma}\label{ord-plain} 
Let $G$ be a group and $C$ be an $\mathtt{N}$-class. $C$ is of plain type if and only if 
 the elements in $C$ have the same order. 
\end{lemma}
\begin{proof}  Let $C$ be of plain type. If $x,y\in C$, then we have $\langle x\rangle=\langle y\rangle$ and thus $o(x)=o(y).$ Assume, conversely, that  the elements in $C$ have the same order. Let $C=[x]_{\mathtt{N}},$ for some $x\in G.$ Since we know that $[x]_{\diamond}\subseteq [x]_{\mathtt{N}},$ we need only to show that $[x]_{\mathtt{N}}\subseteq [x]_{\diamond}$.
Let $y\in [x]_{\mathtt{N}}$. Then $\{x,y\}\in E$ and $o(x)=o(y)$.
Then we deduce $y\diamond x$. 
\end{proof}

We observe that a class of plain type does not have restriction on the order of its elements. Indeed, consider $G\coloneqq \langle a \rangle \times \langle b \rangle$ with $o(a)=2$, $o(b)=k\in \mathbb{N}$. It is easily checked that $[b]_{\mathtt{N}}$ is a class of plain type whose elements have order $k.$

Thanks to Proposition \ref{S>1} and Lemma \ref{ord-plain}, we observe that $\mathcal{S}$ is of compound type if and only if $\mathcal{S}\not=\{1\}.$
We now characterize the $\mathtt{N}$-classes of $G$ different from $\mathcal{S}$ and of compound type. Our  result is very little more than 
 \cite[Proposition 5]{Cameron_2}. We give full details for two reasons. First, some unspecified conditions in  \cite{Cameron_2} are presumably at the origin of a gap that we are going to correct later (see the comments to Proposition \ref{CarachetisationN-classes_2}). Second, we need to clearly introduce some notation for the sequel.
	
	\begin{prop}{\rm \cite[Proposition 5]{Cameron_2}}
		\label{propC_y}
		Let $G$ be a group and $C$ be an $\mathtt{N}$-class of $G$, with $C\neq \mathcal{S}$. The following facts are equivalent:
		\begin{enumerate}
		\item[$(i)$] $C$ is a compound-class;
		\item[$(ii)$] If $y$ is an element of maximum order in $C$, then $o(y)=p^r$ for some prime number $p$ and some integer $r \geq 2$. Moreover, there exists $s\in [r-2]_{0} $ such that $$C= \left\lbrace z \in \langle y\rangle \, |\, p^{s+1}\leq o(z)\leq p^r \right\rbrace. $$ 
		\end{enumerate} 	
		In particular the number of $\diamond$-classes into which a compound $\mathtt{N}$-class $C$ splits is $r-s \geq 2$. The orders of the elements in those $\diamond$-classes are given by $p^{s+1},p^{s+2}, \dots, p^r$ and the sizes of those $\diamond$-classes are $\phi(p^{s+1}), \phi(p^{s+2}),\dots, \phi(p^r)$ respectively.
		The ordered list $(p,r,s)$ is uniquely determined by $C$.
			
	\end{prop}
	
	\begin{proof} $(i)\Rightarrow (ii)$ Let $C$  be a compound-class. We first claim that if $x,y\in C$ are two  elements of distinct order, then those 
	orders are suitable powers, with positive integer exponent, of the same prime. Let $x,y\in C$ with $o(x)\neq o(y)$. Then, renaming if necessary, we have 	
	 $o(x)<o(y)$. In particular, $x\neq y$. Moreover, both $x$ and $y$ are different from $1$ as $C \not= \mathcal{S}=[1]_{\mathtt{N}}$.
	Since  $N[x]=N[y]$, $x$ and $y$ are joined and $o(x)$ is a proper divisor of $o(y)$. Let then $k\coloneqq o(x)$ and $kl\coloneqq o(y)$ for some $k,l\geq 2$. We show that $k$ and $l$ are power of the same prime. It suffices to show that if $p$ is a prime dividing $k$, then the only prime dividing $l$ is $p$ itself. 		
		Assume, by contradiction,  that  we have a  prime $p\mid k$ and a prime $q\neq p$ such that $q\mid l.$ Then $(kq)/p=(k/p)q\mid kl=o(y)$ and thus there exists an element $z\in \langle y\rangle$ with $o(z)=(k/p)q$. Now $z\in N[y]=N[x]$. But  $o(z)\nmid o(x)$ and  $o(x)\nmid o(z)$. Indeed assume that $o(z)\mid o(x)$. Then $(k/p)q\mid k=p(k/p)$ which implies the contradiction $q\mid p$. Assume next $o(x)\mid o(z)$. Then $k\mid (k/p)q$ and thus $p(k/p)\mid q(k/p)$ which implies the contradiction $p\mid q.$
	
		
 Choose now $x\in C$ such that $o(x)$ is minimum in $C$ and $y\in C$ such that $o(y)$ is maximum in $C$. Since $C$ is compound, by Lemma \ref{ord-plain}, 
those orders are distinct. Moreover, since $C\neq \mathcal{S}$, those orders are different from $1$. Hence we have $o(x)=p^{s+1}$ for some integer $s\geq 0$ and $o(y)=p^{r}$ for some integer $r\geq s+2$. Pick $z\in C$. Since $z, y\in C$,  $y$ and $z$ are joined and, since $o(z)\leq o(y)$, we have that $z$ is a power of $y$. Moreover $o(z)\geq p^{s+1}$. This shows that $$C\subseteq \{z\in \langle y\rangle: p^{s+1}\leq o(z)\leq p^r\}.$$ 
		Let next $z\in \langle y\rangle$ with $ o(z)\geq p^{s+1}=o(x)$. Then $z\in N[y]=N[x]$ so that $z=x\in C$ or $z$ and $x$ are joined. In this last case, we necessarily have that $x$ is power of $z$. Moreover, $x$ and $z$ are of prime power order so that, by Lemma \ref{NdipotenzeDip}, we obtain $N[x]\supseteq N[z]\supseteq N[y]=N[x]$ and thus $N[x]= N[z]$, so that $z\in C.$ Thus $C= \{z\in \langle y\rangle: p^{s+1}\leq o(z)\leq p^r\}.$
\smallskip
		
 $(ii)\Rightarrow (i)$ In $C$ there exist elements with different order. Thus, by Lemma \ref{ord-plain}, $C$ cannot be plain and hence it is compound.
	\end{proof}
\begin{definition}\label{DefCompound}{\rm If $C\neq \mathcal{S}$ is an $\mathtt{N}$-class  of compound type for a group $G$, then we call an element $y\in C$ of maximum order a \emph{root} of $C$. Moreover, we call the ordered list
$(p,r,s),$ described in Proposition \ref{propC_y}, the \emph{parameters of} $C$. \\ Recall that $p$ is a prime number, $r$  and $s$ are integer with $r\geq 2$ and $ s\in [r-2]_0.$}
 \end{definition}
 
\subsection{Examples of classes of compound type}	

In  \cite[Proposition 2]{Cameron_1}, Cameron and Ghosh prove, translated within our language, that if $C$ is a compound class in an abelian group $G$, then $C=\mathcal{S}$.
This is not true, in general, for non-abelian groups, as shown by the following example.
	
	\begin{example}\label{ex-comp}
	{\rm The $\mathtt{N}$-classes of the elements of order $4$ in $S_4$ are of compound type with parameters $(p,r,s)=(2,2,0)$. For instance, $[(1234)]_{\mathtt{N}}= \{ (1234), (1432), (13)(24) \}$ is the union of the two $\diamond$-classes $[(1234)]_{\diamond}= \{ (1234), (1432)\}$ and $[(13)(24)]_{\diamond}= \{ (13)(24) \}$.}
	\end{example}
	The above example shows that $s=0$ can occur as a parameter of a compound class. That possibility,  wrongly denied in \cite{Cameron_2}, is not sporadic at all. We can  indeed easily construct an infinite family of groups admitting
	$\mathtt{N}$-classes of compound type with $s=0$. Recall that the dihedral group $D_{2n}$ of order $2n$, for $n\geq 2$, is defined by 
	\begin{equation}\label{diedrale}
	 D_{2n}\coloneqq\langle a,b \, |\ a^{n}=1=b^2,\  b^{-1}ab=a^{-1} \rangle.
	\end{equation}
	
Consider now the choice $n=p^r$ for $p$ a prime and $r\in \mathbb{N}$, with $r\geq 2$. It is easily checked that $\mathcal{P}(D_{2p^r})$ is composed by the complete graph $K_{p^r}$ on the vertices in $\langle a\rangle$ and $p^r$ further vertices given by the involutions in $D_{2n}\setminus \langle a\rangle$ joined only with $1$. Therefore we have that $[a]_{\mathtt{N}}=\langle a \rangle \setminus \{1\}$ is of compound type, since it contains elements of different orders. Moreover,  its parameters are $(p,r,0)$. 

The construction of an infinite family of groups admitting
		$\mathtt{N}$-classes of compound type with $s=1$ is a bit more tricky. Consider, for $n\geq 4$, the quasidihedral group of order $2^n$ defined by
		$$QD_{2^n}\coloneqq \langle a,b\,|\ a^{2^{n-1}}= b^2=1,\  b^{-1}ab = a^{-1+2^{n-2}}\rangle$$
		(see, for instance, \cite[Satz I.14.9]{hup}).
		Then it is easily checked that its proper power graph $\mathcal{P}^*(QD_{2^n})$ is composed by a clique of order $2^{n-1}-1$ given by the graph induced by $\langle a \rangle$, $2^{n-3}$ triangles having the vertex $a^{2^{n-2}}$ in common and all the $2^{n-2}$ remaining vertices isolated. Then the class $[a]_{\mathtt{N}}$ is compound with parameters $(2,n-1, 1)$. In Figure \ref{PPG_quasidiedrale} it is shown the case $n=4$.

	\begin{figure}
		\centering
		\begin{tikzpicture}[line cap=round,line join=round,>=triangle 45,x=1.0cm,y=1.0cm]
			\clip(5.,-5.) rectangle (11.25,4.);
			\draw [line width=1.5pt] (5.5,1.5)-- (7.716485217533858,3.111762854064654);
			\draw [line width=1.5pt] (7.716485217533858,3.111762854064654)-- (8.,0.38128626753482253);
			\draw [line width=1.5pt] (8.,0.38128626753482253)-- (5.5,1.5);
			\draw [line width=1.5pt] (8.28116548812265,3.1057066432285088)-- (8.,0.38128626753482253);
			\draw [line width=1.5pt] (8.28116548812265,3.1057066432285088)-- (10.5,1.5);
			\draw [line width=1.5pt] (10.5,1.5)-- (8.,0.38128626753482253);
			\draw [line width=1.5pt] (6.198062264195162,-0.48648121070029315)-- (9.801937735804838,-0.4864812107002938);
			\draw [line width=1.5pt] (6.198062264195162,-0.48648121070029315)-- (10.246979603717467,-2.436337035063941);
			\draw [line width=1.5pt] (6.198062264195162,-0.48648121070029315)-- (9.,-4.);
			\draw [line width=1.5pt] (6.198062264195162,-0.48648121070029315)-- (7.,-4.);
			\draw [line width=1.5pt] (5.753020396282533,-2.43633703506394)-- (9.801937735804838,-0.4864812107002938);
			\draw [line width=1.5pt] (5.753020396282533,-2.43633703506394)-- (8.,0.38128626753482253);
			\draw [line width=1.5pt] (5.753020396282533,-2.43633703506394)-- (10.246979603717467,-2.436337035063941);
			\draw [line width=1.5pt] (5.753020396282533,-2.43633703506394)-- (9.,-4.);
			\draw [line width=1.5pt] (8.,0.38128626753482253)-- (10.246979603717467,-2.436337035063941);
			\draw [line width=1.5pt] (8.,0.38128626753482253)-- (9.,-4.);
			\draw [line width=1.5pt] (8.,0.38128626753482253)-- (7.,-4.);
			\draw [line width=1.5pt] (9.801937735804838,-0.4864812107002938)-- (9.,-4.);
			\draw [line width=1.5pt] (9.801937735804838,-0.4864812107002938)-- (7.,-4.);
			\draw [line width=1.5pt] (10.246979603717467,-2.436337035063941)-- (7.,-4.);
			\draw [line width=1.5pt] (8.,0.38128626753482253)-- (6.198062264195162,-0.48648121070029315);
			\draw [line width=1.5pt] (6.198062264195162,-0.48648121070029315)-- (5.753020396282533,-2.43633703506394);
			\draw [line width=1.5pt] (5.753020396282533,-2.43633703506394)-- (7.,-4.);
			\draw [line width=1.5pt] (7.,-4.)-- (9.,-4.);
			\draw [line width=1.5pt] (9.,-4.)-- (10.246979603717467,-2.436337035063941);
			\draw [line width=1.5pt] (10.246979603717467,-2.436337035063941)-- (9.801937735804838,-0.4864812107002938);
			\draw [line width=1.5pt] (9.801937735804838,-0.4864812107002938)-- (8.,0.38128626753482253);
			\begin{scriptsize}
				\draw [fill=celesteTol] (7.,-4.) circle (2.5pt) 
				node[anchor=north east] {$a$};
				\draw [fill=celesteTol] (9.,-4.) circle (2.5pt) node[anchor=north west] {$a^7$};
				\draw [fill=celesteTol] (10.246979603717467,-2.436337035063941) circle (2.5pt) node[anchor=north west] {$a^6$};
				\draw [fill=celesteTol] (9.801937735804838,-0.4864812107002938) circle (2.5pt) node[anchor=south west] {$a^5$};
				\draw [fill=gialloTol] (8.,0.38128626753482253) circle (2.5pt) node[anchor=east] {$a^4\quad$};
				\draw [fill=celesteTol] (6.198062264195162,-0.48648121070029315) circle (2.5pt) node[anchor=south east] {$a^3$};
				\draw [fill=celesteTol] (5.753020396282533,-2.43633703506394) circle (2.5pt) node[anchor=east] {$a^2$};
				\draw [fill=rosaTol] (10.5,1.5) circle (2.5pt) node[anchor=north west] {$a^7b$};
				\draw [fill=rosaTol] (8.28116548812265,3.1057066432285088) circle (2.5pt) node[anchor=south west] {$a^3b$};
				\draw [fill=verdeacquaTol] (7.716485217533858,3.111762854064654) circle(2.5pt) node[anchor=south] {$a^5b$};
				\draw [fill=verdeacquaTol] (5.5,1.5) circle(2.5pt) node[anchor=east] {$ab$};
				\draw [fill=bluTol]
				(6.393600271878011,2.59961904885862) circle(2.5pt) node[anchor=south east] {$a^2b$};
				\draw [fill=verdeTol] (9.605706643228508,2.6001207794121743) circle (2.5pt) node[anchor=north west] {$a^4b$};
				\draw [fill=violaTol] (10.724420375693686,0.10012077941217123) circle (2.5pt) node[anchor=north west] {$a^6b$};
				\draw [fill=rossoTol]
				(5.275579624306314,0.1001207794121709) circle(2.5pt) node[anchor=north east] {$b$};;
				
			\end{scriptsize}
		\end{tikzpicture}
	 \caption{$\mathcal{P}^*(QD_{16})$: vertices with different colours are in distinct $\mathtt{N}$-classes.}
	 \label{PPG_quasidiedrale}
	\end{figure}
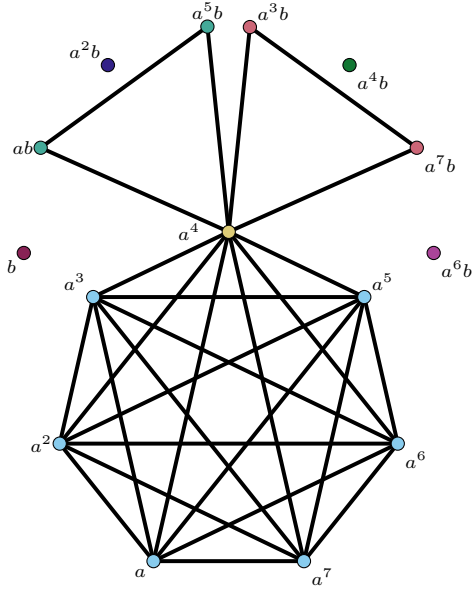

\subsection{Critical classes}\label{sect-crit-class}
	The Moore closure operator defined in Section \ref{Moore-sec} is a very useful tool for the study of $\mathtt{N}$-classes. 
	The proposition below  is essentially extracted from \cite{Cameron_2}. We state it formally and  prove it making use of the neighbourhood closure  operator.	
	\begin{prop}\label{CarachetisationN-classes_1}
		Let $G$ be a group and $C\neq \mathcal{S}$ be an $\mathtt{N}$-class of compound type, root $y \in G$ and parameters $(p,r,s)$. Then  $|C|=p^r-p^s$ and $\hat{C}=\langle y \rangle $. In particular, $|\hat{C}|=p^r$. 
	\end{prop}

		\begin{proof} By Proposition \ref{propC_y}, we have $C \subseteq \langle y\rangle=:Y$ and $|C|=|Y|-|\left\lbrace z\in Y: o(z)\leq p^s \right\rbrace |=p^r-p^s.$ 				We first show that	$Y \subseteq \hat{C}.$ Let $a \in Y$. Then $a$ is power of $y$ and, by Lemma \ref{NdipotenzeDip}, we have $N[a]\supseteq N[y]$. By Proposition \ref{operatoreChiusura} $(v)$, recalling that $y\in C$, we then deduce
		 $$N[C]= N[y] \subseteq \bigcap_{a \in Y} N[a]=N[Y].$$
			Therefore, by Proposition \ref{operatoreChiusura}$\,(i)$-$(ii)$, we get $Y\subseteq \hat{Y}=N[N[Y]] \subseteq N[N[C]]=\hat{C}$.
			
			We next show that	$Y=\hat{C}$. Suppose, by contradiction, that there exists $u\in \hat{C}\setminus Y$.
			
		Assume first that the order of $u$ is not a power of $p$.
				By Proposition \ref{operatoreChiusura} $(v)$, we have $u \in \hat{C}\subseteq N[y]$. But $u\neq y$, because the order of $y$ is a power of $p$. It follows that $u$ and $y$ are joined. Then $\langle u\rangle>\langle y \rangle$ since $u\notin Y$. Thus there exist $t\in \mathbb{N}_0$ and $m \in \mathbb{N}$ with $m\geq 2$ and $\gcd(m,p)=1$, such that $o(u)=p^{r+t}m$.
				Now note that $o(u^{p^tm})=o(y)=p^r$. Thus $u^{p^tm}$ and $y$ are generators of the unique subgroup of order $p^r$ inside the cyclic group $\langle u \rangle$. As a consequence, there exists $k \in \mathbb{N}$, with $\gcd(k,p)=1$, such that $u^{p^tmk}=y$. Thus $y^p=(u^{p^tmk})^p=u^{p^{t+1}mk}$. Now $o(y^p)=p^{r-1}$ and, since $y$ is a root for $C$, we have  $y^p\in C$, so that $y\mathtt{N} y^p$ holds. It follows that $N[y]=N[y^p]$. We  see that this is impossible considering $w\coloneqq u^{p^{t+1}}$ and showing that $w\in N[y^p]\setminus N[y]$. Note  first that $o(w)=p^{r-1}m\notin\{ p^{r-1}, p^r\}$,  so that $w\notin\{y^p,y\}.$ Since $y^p=w^{mk}$, we have that $w\in N[y^p]$. On the other hand $w\notin N[y],$ because 	$o(y)=p^r \nmid o(w)$ and $o(w) \nmid o(y)$.
				
Assume next that the order of $u$ is a power of $p$. By Proposition \ref{operatoreChiusura}\,$(v)$, we have that $u\in \hat{C}\setminus Y \subseteq N[y]\setminus Y$. Thus $u\neq y$, $u$ is joined to $y$ and necessarily $y$ is a power of $u$. Since $u\notin Y$, we have that $o(u)>o(y)$. Hence, by Lemma \ref{NdipotenzeDip}, we have $N[u]\subseteq N[y]$. Assume that $N[u]=N[y]$. Then $u\in  C\subseteq Y$, a contradiction. Thus the inclusion is proper. As a consequence, there exists $w \in N[y]$ such that $w \notin N[u]$. Thus $u\notin N[w]$, that implies $u \notin \hat{C}=\bigcap_{z\in N[y]} N[z]$, a contradiction.

		\end{proof}	
	
Remarkably, Proposition \ref{CarachetisationN-classes_1} points out  that if $C\neq \mathcal{S}$ is an $\mathtt{N}$-class of compound type, then some strong arithmetic restrictions on $|C|$ and $|\hat{C}|$ arise. As a consequence, an $\mathtt{N}$-class different from $\mathcal{S}$ that does not satisfy those restrictions is necessarily of plain type. However, as we show in the next proposition, those arithmetic restrictions can be satisfied also by an $\mathtt{N}$-class of plain type in a very specific case that we can completely clarify.

		\begin{prop} \label{CarachetisationN-classes_2}	Let $G$ be a group and $C$ be an $\mathtt{N}$-class of plain type. Assume that there exist a prime number $p$ and integers $r\geq 2$ and $s\in [r-2]_0$ such that  $|C|=p^r-p^s$ and $|\hat{C}|=p^r$. 
			
			Then $\hat{C}=C\cupdot \{1\}$, $s=0$ and $C=[y]_{\diamond}$ for some $y\in G$, with $o(y)>1$ not a prime power and such that $\phi(o(y))=p^r-1$.
		\end{prop}
		\begin{proof} By Proposition \ref{operatoreChiusura}\,$(i)$, we have   $\hat{C}\supseteq C\cup\{1\} $ and $\hat{C}$ is a union of $\mathtt{N}$-classes and hence of $\diamond$-classes. 
		Observe that $p^r< 2(p^r-p^s)$. Indeed  that inequality is equivalent to
			$ 2< p^{r-s} $ and, surely, $p^{r-s}\geq p^2> 2. $ It follows that $|\hat{C}|< 2|C|$. As a consequence every $\diamond$-class included in $\hat{C}$ and distinct from $C$ must have size smaller than $|C|$. Pick $y\in C$.
			Since $C$ is of plain type we have $C=[y]_{\mathtt{N}}=[y]_{\diamond}$. Note that, since $|C|=p^s(p^{r-s}-1)\geq 3$ and $[1]_{\diamond}=\{1\}$, we have $y\neq 1.$ Then, by Lemma \ref{ord-plain}, the elements in $C$ have all  order equal to $o(y)>1$ and $1\notin C$. In particular, we have $C\cup\{1\}=C\cupdot \{1\}.$
			
			\setcounter{claim}{0}
			\begin{claim}
				\label{claim1_1}
				\rm $\hat{C}\setminus [y]_{\diamond}$ cannot contain elements of order greater than or equal $o(y)$.
			\end{claim}
			Assume, by contradiction, that there exists $x\in \hat{C}\setminus [y]_{\diamond}$ such that $o(x)\geq o(y)$. Then $x\neq y$ and,
			by Proposition \ref{operatoreChiusura}\,$(v)$, we have $\hat{C}\subseteq N[y]$, so that $x \in N[y]\setminus [y]_{\diamond}$. Hence, since $o(x)\geq o(y)$, we necessarily have $y\in \langle x \rangle$. Then $o(y)\mid o(x)$, which implies $\phi(o(y))\mid \phi(o(x))$. In particular we have that $\phi(o(y))\leq \phi(o(x))$.
			Since $\hat{C}$ is a union of $\diamond$-classes, we have that $[x]_{\diamond}\subseteq \hat{C}$. It follows that $|[x]_{\diamond}|=\phi(o(x))\geq \phi(o(y))=|C|$, a contradiction. 	 
			\begin{claim}
				\label{claim1_new}
			\rm  $\hat{C}\subseteq \langle y \rangle$. 
			\end{claim}
			Let $x\in \hat{C}$. If $\langle x\rangle=\langle y\rangle$, then $x$ is a power of $y.$ If $\langle x\rangle\neq \langle y\rangle$, then $x\in \hat{C}\setminus [y]_{\diamond}$ so that, by Claim \ref{claim1_1}, we have $o(x)<o(y)$. Now, by Proposition \ref{operatoreChiusura}\,$(v)$, we have $\hat{C}\subseteq N[y]$ so that 
			 $x\in N[y]$ and thus, again, $x$ is a power of $y.$
						
			Hence we have reached the following chain of inclusions:
			\begin{equation}\label{inequali}
				C\cup \{1\}\subseteq \hat{C}\subseteq \langle y \rangle.
			\end{equation}

			We now show that $o(y)$ cannot be a prime power.  Assume, by contradiction, that 
			$o(y)=q^t$, for some prime $q$ and some integer $t\geq1$. 
			
			We claim that $N[z]\supseteq \langle y \rangle$ holds true  for all $z\in N[y]$. Let $z\in N[y]$. Then we have $z\in \langle y\rangle$ or $y\in \langle z\rangle$.
				Assume first that $o(z)\leq o(y).$ Then  $z \in \langle y \rangle$. Now,  by the fact that $o(y)=q^t$, it follows that $\mathcal{P}(\langle y \rangle)$ is a complete graph and thus $\langle y \rangle \subseteq N[z]$.
				Assume next that $o(z)>o(y)$. Then necessarily $y\in \langle z\rangle$ and hence $\langle y \rangle \leq \langle z \rangle \subseteq N[z]$.
				
			Now, by Proposition \ref{operatoreChiusura}\,$(v)$, we know that $\hat{C}= \bigcap_{z\in N[y]} N[z].$ 	
			Thus  $\hat{C}\supseteq \langle y \rangle $ and, by \eqref{inequali},  we deduce $\hat{C}=\langle y \rangle$.
			As a consequence we have $p^r=|\hat{C}|=|\langle y \rangle|=q^t$. But then $q=p$, $t=r$, and hence $p^r-p^s=|C|=\phi(p^r)=p^r -p^{r-1}$ which implies $s=r-1$, a contradiction.
				
				Hence we have $o(y)=m$, for some 
				  integer $m>1$, not a prime power.
				
				\begin{claim}
					$\hat{C}=C \cup \{1\}$.
				\end{claim}
				
				By \eqref{inequali}, we just need to show that $\hat{C}\subseteq C\cup\{1\}$. Let $x\in \hat{C}$. By \eqref{inequali}, we have that  $x=y^k$, for some $k \in \mathbb{N}$. Assume, by contradiction, that $y^k\notin C\cup \{1\}$.  Then $y^k$  is neither a generator nor the identity of the cyclic group $\langle y \rangle$ of order $m$. By Proposition \ref{S>1} applied to $\mathcal{P}(\langle y \rangle )$, there exists $w\in \langle y \rangle $ such that $w\neq y^k$ and $\{y^k,w\}\notin E_{\mathcal{P}(\langle y \rangle)}$. Then we also have that $\{y^k,w\}\notin E_{\mathcal{P}(G)}$. It follows that $w\in N[y]$ while  $y^k\notin N[w]$. Since by Proposition \ref{operatoreChiusura}\,$(v) $ we have $\hat{C}=\bigcap_{z \in N[y]} N[z]$, we deduce that $x=y^k \notin \hat{C}$, a contradiction.
By  $\hat{C}=C\cupdot \{1\}$, we now deduce  $|\hat{C}|=|C|+1$ and thus $p^r-p^s=|C|=|\hat{C}|-1=p^r-1$ which implies $s=0$. As a consequence, we also have $\phi(o(y))=|C|=p^r-1.$
\end{proof}
		
In \cite[pages 782--783]{Cameron_2}, it is claimed  that there is no $\mathtt{N}$-class $C$ of plain type satisfying 
$|C|=p^r-p^s$ and $|\hat{C}|=p^r$, for some prime number $p$ and some $r,s$ integers with $r\geq 2$ and $s\in [r-2]_0.$ That is a mistake,  because it is instead easy to exhibit examples of $\mathtt{N}$-classes of plain type satisfying those arithmetical restrictions. 

\begin{example}\label{controesempio}{\rm Consider $G=D_{30}$ with the notation in \eqref{diedrale}. Let $C\coloneqq[a]_{\mathtt{N}}$.
$C$ contains the element $a$ of order $15$ and $15$ is not a prime power. We remark that, by Proposition \ref{propC_y}, if an $\mathtt{N}$-class $C=[x]_{\mathtt{N}}\not=\mathcal{S}$ is compound, then $o(x)>1$ is a prime power. Thus we deduce that $C$ is of plain type. As a consequence, we have   $C=[a]_{\diamond}$ and $|C|=\phi(15)= 8=3^2-1.$ Defining now $p\coloneqq3$, $r\coloneqq2$, $s\coloneqq0$ we see that $|C|=p^r-p^s$. In order to show that $|\hat{C}|=p^r$, we prove that
 $\hat{C}=C \cupdot \{1\}$.  Since the elements in $C$ have order $15$, we clearly have $1\notin C$, so that $C \cup \{1\}=C \cupdot \{1\}.$
By Proposition \ref{operatoreChiusura}\,$(i)$, we have $\hat{C}\supseteq C \cup \{1\}$. 
Assume, by contradiction, that there exists $x\in \hat{C}\setminus(C \cup \{1\})$. Then we have  $o(x)\notin\{1, 15\}$. In particular, $x\neq a.$ Moreover, by Proposition \ref{operatoreChiusura}\,$(v)$, we have $\hat{C}\subseteq N[a] $ and therefore $x \in \langle a \rangle$. Then, by Proposition \ref{S>1}, there exists $y\in \langle a \rangle\setminus \{x\}$ such that $y$ is not joined to $x$. It follows that $y\in N[a]$ and $x\notin N[y]$, so that  $x \notin \hat{C}=\bigcap_{z \in N[a]} N[z]$, a contradiction.
}

\end{example}

We are now in position to give birth to a crucial and original definition, the main player of our research.
		
\begin{definition}
	\label{DefCriticalClass}
			\rm	Let $\Gamma$ be a power graph. A \emph{critical class} is an $\mathtt{N}$-class $C$ such that $\hat{C}=C \cupdot \{1\}$ and there exist a prime number $p$ and an integer $r\geq 2$, with $|\hat{C}|=p^r$. 
					\end{definition}
			We emphasize that, in order to check if an $\mathtt{N}$-class is critical, one has to make only arithmetical or graph theoretical considerations. No group theoretical consideration is involved.
			We also want to emphasize that both critical classes of plain type and of compound type can arise in a power graph as shown in the Examples \ref{ex-comp} and \ref{controesempio}.
			Note that $\mathcal{S}$ is never critical since $1\in \mathcal{S}$.
				
	It is interesting to note that a compound class $C$ is critical if and only if $C\not= \mathcal{S}$ and $C$ has parameters $(p,r,0)$. This can easily proved as follows.  Assume that $C$  is critical. Then $C\neq \mathcal{S}$ because $\mathcal{S}$ is never critical. Moreover $\hat{C}=C\cupdot\{1\}$ and $|\hat{C}|=p^r$ for some prime $p$ and integer $r\geq2$. Hence we have $|C|=p^r-1$. Then, by Proposition \ref{CarachetisationN-classes_1}, the parameters of $C$ are $(p,r,0)$. 
	Conversely, assume that $C\neq \mathcal{S}$ and the parameters of $C$ are $(p,r,0)$. By Proposition \ref{CarachetisationN-classes_1}, we have $|\hat{C}|=p^r$ and $|C|=p^r-1$. Now clearly $1\notin C$, otherwise $C=\mathcal{S}$. On the other hand, by Proposition \ref{operatoreChiusura}\,$(i)$, we have  $\hat{C}\supseteq C\cup \{1\}$. It follows that $\hat{C}=C \cupdot \{1\}$. Hence $C$ is critical.
		
	We stress that a critical class $C$ is an $\mathtt{N}$-class, different from $\mathcal{S}$, which we cannot immediately recognize as plain or compound by arithmetical considerations of its size and the size of its closure. On the other hand, if we exclude those classes the recognition is easy.
			
\begin{prop}\label{appoggio} Let $G$ be a group and $C\ne \mathcal{S}$ be a non-critical $\mathtt{N}$-class. Then $C$ is compound if and only if there exist a prime number $p$, an integer $r\geq 2$ and an integer $s\in [r-2]_0$ such that  $|C|=p^r-p^s$ and $|\hat{C}|=p^r$. 
\end{prop}
\begin{proof} If $C$ is compound, by Proposition \ref{CarachetisationN-classes_1}, there exist a prime $p$, an integer $r\geq 2$ and an integer $s\in [r-2]_0$ such that  $|C|=p^r-p^s$ and $|\hat{C}|=p^r$. Conversely, Proposition \ref{CarachetisationN-classes_2} shows that if such $p$, $r$ and  $s$ exist and $C$ is plain, then $C$ is critical, a contradiction.
\end{proof}

		For a better final insight on $\mathtt{N}$-classes we need the following result by Feng, Ma and Wang \cite{Ma et al}.
		
		\begin{prop}{\rm \cite[Lemma 3.5]{Ma et al} }
			\label{propLatiNclassi}
			Let $G$ be a group and $x,y \in G$. Let $[x]_{\mathtt{N}}$ and $[y]_{\mathtt{N}}$ be two distinct $\mathtt{N}$-classes different from $\mathcal{S}$. If $\langle x \rangle < \langle y \rangle $, then $|[x]_{\mathtt{N}}|\leq |[y]_{\mathtt{N}}|$, with equality if and only if both the following two conditions hold:
			\begin{enumerate}
				\item[$(i)$] Both $[x]_{\mathtt{N}}$ and $[y]_{\mathtt{N}}$ are of plain type;
				
				\item[$(ii)$] $o(y)=2\,o(x)$ and $o(x)\geq 3$ is odd.
			\end{enumerate} 
		\end{prop}
		
		We can now state and prove  a result that allows to recognize if a critical class is plain or compound, by  purely graph theoretical considerations, when the star class is trivial. Such a result will be crucial for the proof of the Main Theorem.
		
		\begin{prop}
			\label{LemmaCriticalClassMio} Let $G$ be a group with $\mathcal{S}=\{1\}$ and $C=[y]_{\mathtt{N}}$ be a critical class. Then $C$ is of plain type if and only if there exists $x\in G\setminus \hat{C}$ such that $|[x]_{\mathtt{N}}|\leq |C|$ and  $\{x,y\} \in E$.
		\end{prop}

		\begin{proof} 
			Note that we have $C\ne \mathcal{S}$, because $\mathcal{S}$ is never critical.
			Assume first that $C$ is of plain type. Then $\hat{C}$ is formed by the generators of $\langle y \rangle$ and by $1$.
			By Proposition \ref{CarachetisationN-classes_2}, we have $o(y)=m>1$ not a prime power. In particular, $m$ is not a prime and thus $m>\phi(m)+1.$ As a consequence, $\langle y \rangle\setminus \hat{C}\neq \varnothing.$ Pick $x\in \langle y \rangle\setminus \hat{C}$. Then we have $\langle x \rangle < \langle y \rangle$ and then also $\{x,y\}\in E$. Note that, since $x\neq 1$, we have $[x]_{\mathtt{N}}\ne \mathcal{S}=\{1\}$. Moreover, since $x$ does not generate $\langle y \rangle$, we also have $[x]_{\mathtt{N}}\ne [y]_{\mathtt{N}}=[y]_{\diamond}$. Hence, by Proposition \ref{propLatiNclassi}, we deduce $|[x]_{\mathtt{N}}|\leq |C|$.
			
			Assume next that there exists  $x\in G\setminus \hat{C}$ such that $|[x]_{\mathtt{N}}|\leq |C|=|[y]_{\mathtt{N}}|$ and $\{x,y\}\in E$.
			Note that
			$x\ne 1$ since $x \notin \hat{C}=C\cupdot \{1\}$. 
			As a consequence $[x]_{\mathtt{N}}\ne \mathcal{S}=\{1\}$.   Observe next that $[x]_{\mathtt{N}}\ne [y]_{\mathtt{N}}$ holds, otherwise we would have $x\in C\subseteq \hat{C}$. In particular,  we have $\langle x \rangle \not= \langle y \rangle$.
			Since $\{x,y\}\in E$, it follows that $\langle y \rangle < \langle x \rangle$ or $\langle x \rangle < \langle y \rangle$.
				If $\langle y \rangle < \langle x \rangle$, then, by Proposition \ref{propLatiNclassi}, we deduce $|[y]_{\mathtt{N}}|\leq |[x]_{\mathtt{N}}|$ and hence $|[x]_{\mathtt{N}}|=|[y]_{\mathtt{N}}|$. Thus, by Proposition \ref{propLatiNclassi}, $C=[y]_{\mathtt{N}}$ is of plain type.
			If $\langle x \rangle < \langle y \rangle$, then $x\in \langle y \rangle $.
			Suppose, by contradiction, that $C$ is of compound type. Let $z$ be a root of $C$. Then, by Propositions \ref{propC_y} and  \ref{CarachetisationN-classes_1}, we get $x\in \langle y \rangle\leq \langle z \rangle=\hat{C}$, a contradiction.

\end{proof}
		
	\section{The reconstruction of the directed power graph}	
		We now describe the $\diamond$-classes inside the directed power graph, exploiting some facts  observed in \cite{Cameron_2}. The following lemma, together with other previous results,  paves the road for the effective reconstruction of the directed power graph from its undirected counterpart. 
		\begin{lemma}
			\label{lemmaClassiDiamond}Let $G$ be a group and let $X, Y$ be two distinct $\diamond$-classes. In $\vec{\mathcal{P}}(G)$ the following facts hold:
			\begin{enumerate}
				
				\item[$(i)$] The subdigraph induced by a $\diamond$-class is a complete digraph.
				
				\item[$(ii)$]  If there is at least one arc directed from $X$ to $Y$, then $(x,y)$ is an arc  for all $x\in X$ and $y\in Y$. Moreover, there is no arc directed from $Y$ to $X$.
				\item[$(iii)$] Let $X$ and $Y$ be joined and $|X|>|Y|$. Then there is an arc directed from $X$ to $Y$.
				\item[$(iv)$] Let $X$ and $Y$ be joined and $1\not=|X|=|Y|$. There is an arc directed from $X$ to $Y$, if and only if there exists an involution $\tau\in G$ such that $[\tau]_\diamond$ is joined with $X$.
				\item[$(v)$] Let $X$ and $Y$ be joined and $1=|X|=|Y|$. Then one of them is $[1]_\diamond$, the other is $[\tau]_\diamond$, with $\tau\in G$ an involution, and $(\tau, 1)$ is the only arc between the two $\diamond$-classes. 
			\end{enumerate}
		\end{lemma}
		
		\begin{proof}  Recall that $A$ denotes the arc set of $\vec{\mathcal{P}}(G)$.
			\begin{enumerate}
				
				\item[$(i)$] This is obvious by the definition of the relation $\diamond$.
				
				\item[$(ii)$] Note first that, if $x \in X$ and $y\in Y$, then we cannot have in $A$ both the arcs $(x,y)$ and $(y,x)$, otherwise we would have $x\diamond y$ and then $X=Y$.
				Suppose now that there exist $\bar{x}\in X$ and $\bar{y}\in Y$ such that $(\bar{x},\bar{y})\in A$. 
				Pick $x\in X$ and $y\in Y$. Then $y$ is a power of $\bar{y}$, which is power of $\bar{x}$, which in turn is a power of $x$. Hence $(x, y)\in A$. 	\smallskip
				
\item[$(iii)$] Since the classes are joined there exist joined vertices $x \in X$ and $y \in Y$. 
				Since $\phi(o(x))=|X|>|Y|=\phi(o(y))$, we deduce that $o(y)\mid o(x)$ and so $y$ is a power of $x.$
				
				\smallskip

				\item[$(iv)$] Let $X=[x]_{\diamond}$ and $Y=[y]_{\diamond}$ for some $x,y \in G$.
				Assume first that $(x,y) \in A$.
				Then $y\in \langle x\rangle$ and $o(y) \mid o(x)$. Now $|X|=|Y|$ implies $\phi(o(x))=\phi(o(y))$.
				This implies that $o(x)=2o(y)$, with $o(y)$ odd, since otherwise we would have $o(x)=o(y)$ and then also the contradiction $X=Y$. Observe that $\tau\coloneqq x^{o(y)}$ is an involution and that $x\neq \tau$ because $|[\tau]_{\diamond}|=1\neq |[x]_{\diamond}|.$			
				It follows that $(x, \tau) \in A$ and thus $[\tau]_\diamond$ and $X$ are joined.\\
				Conversely assume that there exists an involution $\tau\in G$ such that $[\tau]_\diamond$ and $X$ are joined. Then we have $|X|>|[\tau]_\diamond|=1$ and thus, by $(iii)$, $(x,\tau)\in A$. We show that $(x,y)\in A$. Assume, by contradiction, that $(y,x)\in A$. As before one obtains $o(y)=2o(x)$, with $o(x)$ odd, against the fact that $2=o(\tau)\mid o(x)$.

				\item[$(v)$] A $\diamond$-class of size one contains either $1$ or an involution, and involutions are never joined. The result therefore follows from the definition of $\vec{\mathcal{P}}(G)$.
			\end{enumerate}
		\end{proof}
		
		\subsection{The Main Theorem}\label{dimmain}

We now pass to prove our main theorem. First we need to be precise about our terminology. It seems that, in the literature, a clear definition of reconstruction was missing.

		We say that we can
		reconstruct the directed power graph from  a power graph $\Gamma=(V,E)$ if we are able, by purely arithmetical or graph theoretical considerations, without taking into account any group theoretical information, to do one of the following:
		\begin{itemize}
		\item prove that there exists a unique group $G$ such that $\Gamma=\mathcal{P}(G)$ and exhibit such $G$;
\item exhibit a digraph $\vec{\Gamma}$ isomorphic to $\vec{\mathcal{P}}(G)$ for all those $G$ such that $ \Gamma=\mathcal{P}(G)$.	

		\end{itemize}
		Note that in the first case, $G$ is uniquely determined and exhibited and thus we can clearly also exhibit its directed power graph. In the second case  there could be many groups $G$ realizing $ \Gamma=\mathcal{P}(G)$, and usually one is not able to explicitly exhibit them. The point is that, whatever those groups are, we require to be able to show a directed graph which, up to isomorphisms of directed graphs, equals the directed power graphs of all those groups. In particular, the directed power graphs of all those $G$ will be isomorphic.
		
		For the proof, we are going to use some methods from the proof of \cite[Theorem 2]{Cameron_2}, correcting the mistake about critical classes and filling in some missing details. Of course, since we are proving a stronger result, the architecture of the proof and some parts of it are completely original. 
		
	\begin{main}\label{UPG-DPG_generalised} We can reconstruct the directed power graph from any power graph.
		\end{main}
		
		\begin{proof} Let $\Gamma=(V,E)$ be a power graph and let $n\coloneqq|V|$. Since $\Gamma$ is a power graph, there exists a group $G$ such that $\Gamma=\mathcal{P}(G)$ and $V=G.$ If $n=1$, then $G=1$; if $n=2$ then $G\cong C_2$. Hence in those cases $G$ is uniquely determined and exhibited.
Assume then that $n\geq 3$.
			
			We consider the size of the set  $\mathcal{S}$ of the star vertices in $\Gamma.$ By Proposition \ref{S>1}, if $|\mathcal{S}|>1$ the following
			three possibilities arise, each of them leading to a unique and exhibited group $G$:
			\begin{itemize}
				\item $|\mathcal{S}|=n$. In this case  the only possibility is $G\cong C_n$ with $n$ a prime power. 
				
				\item $|\mathcal{S}|=1+\phi(n)\neq n$. Here we have  the only possibility  $G\cong C_n$, with $n$ not a prime power.
				
				\item $|\mathcal{S}|=2$. Note that we are not in one of the previous cases because $n\geq 3$ implies $2\ne n$, and  $2\neq 1+\phi(n).$ 
				Here, $G$ is the generalized quaternion group of order $n$. 				
			\end{itemize}
		 We now study the case $|\mathcal{S}|=1$. Then $\mathcal{S}=\{1\}$, and we recognize which vertex of $\Gamma$ is the identity element $1$ of the group $G$.
		This is the genuine interesting case to deal with and it needs the whole machinery of the paper.
		
		Let $\mathcal{K}$ be the partition of $V\setminus\{1\}$ into $\mathtt{N}$-classes. We show that, given a class in $ \mathcal{K}$, we can decide if it is of plain or compound type by arithmetical or graph theoretical considerations, without taking into account any group theoretical information.
		
		Pick then $C=[y]_{\mathtt{N}}\in \mathcal{K}$. Assume first that $C$ is critical. Then, by Proposition \ref{LemmaCriticalClassMio}, $C$ is plain if and only if there
		exists $x\in V\setminus \hat{C}$ such that $|[x]_{\mathtt{N}}|\leq|C|$ and $\{x,y\}\in E$. Assume next that $C$ is not critical.  Then, by Proposition \ref{appoggio}, $C$ is compound if and only if there exist a prime $p$ and integers $r\geq 2$ and $s\in [r-2]_0$ such that  $|C|=p^r-p^s$ and $|\hat{C}|=p^r$. 
		
		In $\mathcal{K}$, denote by $\mathcal{K}_{\mathfrak{P}}$ the set of plain classes and by $\mathcal{K}_{\mathfrak{C}}$ the set of compound classes. Of course, we have $\mathcal{K}=\mathcal{K}_{\mathfrak{P}}\cupdot \mathcal{K}_{\mathfrak{C}}.$
		
		Let $C\in \mathcal{K}_{\mathfrak{C}}$. By Proposition \ref{CarachetisationN-classes_1},  we have parameters  $(p,r,s)$ associated with $C$, where $p$ is a prime, $r\geq 2$ and $s\in [r-2]_0$. Recall that  $|C|=p^r-p^s$ and $|\hat{C}|=p^r$ and note that 
		$$p^r-p^s=\sum_{i=s+1}^r (p^i-p^{i-1})=\sum_{i=s+1}^r \phi(p^i).$$ 
	We now partition $C$, arbitrarily, into $r-s\geq 2$ subsets $X_i(C)$ of sizes $\phi(p^i)$, for $s+1\leq i \leq r$.
		Let $\mathcal{K}_{C}\coloneqq\{X_i(C): s+1\leq i \leq r\}$ be the obtained partition of $C$. 
				
		By Proposition \ref{propC_y}, we  know that $C\in \mathcal{K}_{\mathfrak{C}}$ also admits the partition $\diamond_{C}$ given by the  $r-s$ $\diamond$-classes of $C$, and that the sizes of those $\diamond$-classes  are $\phi(p^i)$, for $s+1\leq i \leq r$. Thus $\mathcal{K}_{C}$ and $\diamond_{C}$ are two partitions of $C$ formed by the same number of subsets of the same sizes. As a consequence, there clearly exists $\psi_C\in S_{C}$ such that $\psi_C(X_i(C))$ is a $\diamond$-class of $C$ of size $\phi(p^i)$, for all $s+1\leq i \leq r$, and $\diamond_{C}=\{\psi_C(X):X\in\mathcal{K}_{C} \}=\psi_C(\mathcal{K}_{C}) .$
		
		Define next the partition $\mathcal{K}_{\diamond}$ of $V\setminus\{1\}$ given by $\mathcal{K}_{\diamond}\coloneqq\mathcal{K}_{\mathfrak{P}}\cup \bigcup_{C\in \mathcal{K}_{\mathfrak{C}}} \mathcal{K}_{C}$. 
		By what was shown above, we have that there exists $\psi_{\mathfrak{C}}\in \times_{_{C\in\mathcal{K}_{\mathfrak{C}}}} S_{C}$ such that $\psi_{\mathfrak{C}}( \bigcup_{C\in \mathcal{K}_{\mathfrak{C}}} \mathcal{K}_{C}) $ is the partition of the set $ \bigcup_{C\in \mathcal{K}_{\mathfrak{C}}} C$ into its $\diamond$-classes. Completing the permutation $\psi_{\mathfrak{C}}$ to a permutation of $V$ which fixes the $\diamond$-classes in $\mathcal{K}_{\mathfrak{P}}$ and $1$, we obtain $\psi\in \times_{_{C\in\mathcal{K}\cup \{\mathcal{S}\}}} S_{C}$ such that $\psi(\mathcal{K}_{\diamond})$ is the partition of $V\setminus \{1\}$ in $\diamond$-classes.
		
		Now, by  Proposition \ref{N-classi_automorphism}, the group $\times_{_{C\in\mathcal{K}\cup \{\mathcal{S}\}}} S_{C}$ is a group of automorphisms for the power graph $\Gamma.$
		Hence, what shown above, proves that $\mathcal{K}_{\diamond}$ is, up to the graph isomorphism $\psi$, the partition of $V\setminus \{1\}$ in $\diamond$-classes.
		
		We are now ready to define a set of arcs $A_{\vec{\Gamma}}\subseteq V\times V$.
		For every $\{x,y\} \in E$ we are going to set $(x,y)\in A_{\vec{\Gamma}}$ or $(y,x)\in A_{\vec{\Gamma}}$ (or both). 
		Let $\{x,y\}\in E$.
		Assume first that $1\in \{x,y\}$, say $y=1$. Then we set $(x,1)\in A_{\vec{\Gamma}}$.
		Assume next that $1\notin \{x,y\}$.
		If there exists $C\in \mathcal{K}_{\diamond}$ such that $x,y\in C$, then we set both $(x,y)\in A_{\vec{\Gamma}}$ and $(y,x)\in A_{\vec{\Gamma}}$.
		If $x\in X$, $y\in Y$, with $X,Y\in \mathcal{K}_{\diamond}$ and $X\neq Y$, then we take our decision through the computation of $|X|$ and $|Y|$.
		If $|X|>|Y|$, then we set $(x,y)\in A_{\vec{\Gamma}}$. Assume next that $|X|=|Y|\ne 1$. Then we also have $|\psi(X)|=|\psi(Y)|\ne 1$ and, by definition of $\psi$,  $\psi(X)$ and $\psi(Y)$ are $\diamond$-classes. Thus  Lemma \ref{lemmaClassiDiamond} implies that there exists $Z\in \mathcal{K}_{\diamond}$ such that $\psi(Z)=[\tau]_\diamond$, with $\tau\in G$ an involution and $\psi(Z)$ is joined with exactly one of $\psi(X)$ or $\psi(Y)$. Hence $|Z|=1$ and, since $\psi$ is a graph isomorphism, $Z$ is joined with exactly one of $X$ or $Y$. We set $(x,y)\in A_{\vec{\Gamma}}$ if $X$ is joined with $Z$, and we set $(y,x)\in A_{\vec{\Gamma}}$ if $Y$ is joined with $Z$.
		Finally we suppose, by contradiction, that $|X|=|Y|=1$. Then, by Lemma \ref{lemmaClassiDiamond}, one of the $\diamond$-classes $\psi(X)$ and $\psi(Y)$ must be $\mathcal{S}=\{1\}$, against the fact that $\psi(X),\psi(Y) \in \mathcal{K}_{\diamond}$, a partition of $V\setminus \{1\}$.
		
		We define now $\vec{\Gamma}\coloneqq(V,A_{\vec{\Gamma}})$.
		We claim that $\psi$ is a digraph isomorphism between $\vec{\Gamma}$ and the directed power graph $\vec{\mathcal{P}}(G)=(G,A)$.
		First observe that $\psi$ is, by definition, a bijection between the vertex sets $V=G$ of the two digraphs.
		We show that $(x,y)\in A_{\vec{\Gamma}}$ if and only if $(\psi(x), \psi(y))\in A$.  Assume first that
		$x,y$ belong to the same $C\in \mathcal{K}_{\diamond}$. This happens if and only if $\psi(x)$ and $\psi(y)$ belong to the same $\diamond$-class  $\psi(C)$.  Hence $(x,y)\in A_{\vec{\Gamma}}$ if and only if $(\psi(x), \psi(y))\in A$,  by the construction above and by Lemma \ref{lemmaClassiDiamond}\,$(i)$.
		Assume next that $x\in X$, $y\in Y$, with $X,Y \in \mathcal{K}_{\diamond}$ and $X\ne Y$. This happens if and only if $\psi(x)\in \psi(X)$ and $\psi(y)\in \psi(Y)$ with $\psi(X) \ne \psi(Y)$ $\diamond$-classes.
		
		Then, by the definition of $A_{\vec{\Gamma}}$ and by Lemma \ref{lemmaClassiDiamond}\,$(iii)$-$(iv)$, one of the following holds:
		\begin{itemize}
			\item $|X|>|Y|$ if and only if $|\psi(X)|>|\psi(Y)|$ and hence $(x,y)\in A_{\vec{\Gamma}}$ if and only if $(\psi(x), \psi(y))\in A$.
			
			\item $|X|=|Y|\ne 1$ and there exists $Z\in \mathcal{K}_{\diamond}$ such that $|Z|=1$ and $Z$ is joined to $X$ if and only if the same is true changing $X, Y$ and $Z$ with, respectively $\psi(X), \psi(Y)$ and $\psi(Z)$. Thus $(x,y)\in A_{\vec{\Gamma}}$ if and only if $(\psi(x), \psi(y))\in A$. 
		\end{itemize}
	
		It remains to consider the arcs in $A_{\vec{\Gamma}}$ incident to $1$. By the construction of $\vec{\Gamma}$, those arcs are the $(x,1)$ for $x\in G\setminus \{1\}$, and obviously $(\psi(x), \psi(1))=(\psi(x),1)\in A$ for all $x\in G\setminus \{1\}$.
			
		\end{proof}
		As a corollary we immediately deduce the following result. It appeared in \cite{Cameron_2}, as the main theorem, with a step of the proof affected by a mistake, as explained in Section \ref{sect-crit-class}. Thanks to the Main Theorem, we can completely confirm its validity.
		
		\begin{corollary}{\rm \cite[Theorem 2 ]{Cameron_2}}\label{cCameron_2}
			\label{UPG-DPG}
			If $G_1$ and $G_2$ are finite groups whose power graphs are isomorphic, then their directed power graphs are also isomorphic.
		\end{corollary}
		
	We emphasize that the Main Theorem  expresses a  stronger result with respect to Corollary \ref{cCameron_2}. In particular, it allows us to give a clear answer to the request, posed in 2022 by P. J. Cameron \cite[Question 2]{GraphsOnGroups}, for a simple algorithm able to reconstruct the directed power graph of a group  from its power graph. Indeed its proof  is  constructive and can be explicitly converted into an algorithm.
The  readers interested into the description of the algorithm and its pseudo-code are referred to the Appendix.
Here, we illustrate how the algorithm works on some enlightening examples and in the framework of a game. This approach should make clear that the Main Theorem expresses quite a surprising result: we can reconstruct the directed version of a power graph, having no knowledge about the possible groups of which it is the power graph.

\subsection{Examples of reconstruction of the directed power graph}\label{ex-recon}

Imagine we play mathematics with a friend.  She has a finite group $G$ in hands and computes the power graph $\Gamma\coloneqq\mathcal{P}(G)$. Then she hides $G$ inside an inaccessible black-box and let us see only $\Gamma.$ She challenges us to guess the shape of $\vec{\mathcal{P}}(G)$.

The game is quite boring if the graph shown has a star set with more than one vertex. In that case we can easily guess what the group $G$ in the black-box  is, by Proposition \ref{S>1}, and then obtain its directed power graph.

So suppose she opts for something like $G=D_{30}$. The graph that she shows us is that in Figure \ref{PGD15}, but without colours.  We put colours to recognize the different nature of the vertices.
The white vertex is clearly the only star vertex, hence it is the identity of the group. The orange vertices are joined only with the identity, so they are involutions.
We split the remaining vertices into $\mathtt{N}$-classes just examining the edges in the graph. We end up with three $\mathtt{N}$-classes, each containing only vertices of one colour: magenta, blue or yellow.
We now want to understand if those classes are of plain or of compound type. 
Let $x$ and $z$ be, respectively, a blue and a yellow vertex. Let also $y$ be one of the magenta vertices.
We start with the easy ones.
Observing the graph we are able to compute the neighbourhood closure of $[x]_{\mathtt{N}}$ and $[z]_{\mathtt{N}}$. We obtain $\widehat{[x]_{\mathtt{N}}}=[x]_{\mathtt{N}}\cup [y]_{\mathtt{N}}\cup \{1\}$ and $\widehat{[z]_{\mathtt{N}}}=[z]_{\mathtt{N}}\cup [y]_{\mathtt{N}}\cup \{1\}$, hence $[x]_{\mathtt{N}}$ and $[z]_{\mathtt{N}}$ are not critical classes. Since we have $|[x]_{\mathtt{N}}|=4$, $|\widehat{[x]_{\mathtt{N}}}|=12$, $|[z]_{\mathtt{N}}|=2$, $|\widehat{[z]_{\mathtt{N}}}|=10$, then neither $[x]_{\mathtt{N}}$ nor $[z]_{\mathtt{N}}$ are of compound type, by Proposition \ref{CarachetisationN-classes_1}. 
Now let us focus on the class $[y]_{\mathtt{N}}$ of magenta vertices.
As before we compute the neighbourhood closure obtaining $\widehat{[y]_{\mathtt{N}}}=[y]_{\mathtt{N}} \cupdot \{1\}$. Evaluating also $|\widehat{[y]_{\mathtt{N}}}|=9=3^2$, we see that $[y]_{\mathtt{N}}$ is a critical class, just recalling Definition \ref{DefCriticalClass}.
In order to understand the nature of $[y]_{\mathtt{N}}$ we seek outside $\widehat{[y]_{\mathtt{N}}}$. It is straightforward to check the following facts for $x$, one of the blue vertices:
\begin{enumerate}
	\item $x\in G\setminus \widehat{[y]_{\mathtt{N}}}$;
	
	\item $4=|[x]_{\mathtt{N}}|\leq|[y]_{\mathtt{N}}|=8$;
	
	\item $\{x,y\}\in E$.
\end{enumerate} 
Therefore, by Proposition \ref{LemmaCriticalClassMio}, we deduce that $[y]_{\mathtt{N}}$ is of plain type. Summarizing, we have
$[y]_{\mathtt{N}}=[y]_{\diamond}$, $[x]_{\mathtt{N}}=[x]_{\diamond}$ and $[z]_{\mathtt{N}}=[z]_{\diamond}$.

 We now have the partition composed by $[y]_{\diamond}$, $[x]_{\diamond}$, $[z]_{\diamond}$ and other sixteen $\diamond$-classes each containing the identity or an involution. That is the partition of $V$ in $\diamond$-classes.

The game is not over yet, because we need to exhibit a digraph, but we are in the homestretch.
We  replace all the edges joining a vertex with the identity with arcs of the type $(v,1)$, for $v \in V\setminus \{1\}$. 
For all the other edges we follow Lemma \ref{lemmaClassiDiamond}. In particular we use $(i)$ and $(iii)$.
Edges between vertices of the same $\diamond$-class are replaced by two arcs, one for both the directions.
The remaining edges are the one between the magenta vertices and the blue or yellow vertices.
Without hesitating we replaced all those edges with arcs starting in the magenta vertices and ending in the blue or yellow vertices, because $|[y]_{\mathtt{N}}|=8$, $|[x]_{\mathtt{N}}|=4$ and $|[z]_{\mathtt{N}}|=2$ hold. 

\begin{figure}
	\centering
	\begin{tikzpicture}[line cap=round,line join=round,>=triangle 45,x=1.0cm,y=1.0cm]
		\clip(-6.,-8.) rectangle (8.,10.);
		\draw [line width=1.pt] (2.993657089646132,-0.5747290551732837)-- (4.641980191492829,0.6246686768274756);
		\draw [line width=1.pt] (2.993657089646132,-0.5747290551732837)-- (5.659959266685525,2.390806431953024);
		\draw [line width=1.pt] (2.993657089646132,-0.5747290551732837)-- (5.871576485073833,4.418302947484353);
		\draw [line width=1.pt] (2.993657089646132,-0.5747290551732837)-- (5.240241307116339,6.356585656658499);
		\draw [line width=1.pt] (2.993657089646132,-0.5747290551732837)-- (3.8751173205819267,7.870507870313641);
		\draw [line width=1.pt] (2.993657089646132,-0.5747290551732837)-- (2.0122468645044806,8.698298684157265);
		\draw [line width=1.pt] (2.993657089646132,-0.5747290551732837)-- (-0.026262835613412694,8.696825546232356);
		\draw [line width=1.pt] (2.993657089646132,-0.5747290551732837)-- (-1.8879349333421296,7.866343175469165);
		\draw [line width=1.pt] (2.993657089646132,-0.5747290551732837)-- (-3.2508694102243165,6.350449518468134);
		\draw [line width=1.pt] (2.993657089646132,-0.5747290551732837)-- (-3.8794025133360366,4.411256359986279);
		\draw [line width=1.pt] (2.993657089646132,-0.5747290551732837)-- (-3.6648551591052874,2.3840678141418907);
		\draw [line width=1.pt] (2.993657089646132,-0.5747290551732837)-- (-2.6443245341800923,0.6194031969412541);
		\draw [line width=1.pt] (2.993657089646132,-0.5747290551732837)-- (-0.9942696548396859,-0.5776109478268854);
		\draw [line width=1.pt] (-0.9942696548396859,-0.5776109478268854)-- (-2.6443245341800923,0.6194031969412541);
		\draw [line width=1.pt] (-0.9942696548396859,-0.5776109478268854)-- (-3.6648551591052874,2.3840678141418907);
		\draw [line width=1.pt] (-0.9942696548396859,-0.5776109478268854)-- (-3.8794025133360366,4.411256359986279);
		\draw [line width=1.pt] (-0.9942696548396859,-0.5776109478268854)-- (-3.2508694102243165,6.350449518468134);
		\draw [line width=1.pt] (-0.9942696548396859,-0.5776109478268854)-- (-1.8879349333421296,7.866343175469165);
		\draw [line width=1.pt] (-0.9942696548396859,-0.5776109478268854)-- (-0.026262835613412694,8.696825546232356);
		\draw [line width=1.pt] (-0.9942696548396859,-0.5776109478268854)-- (2.0122468645044806,8.698298684157265);
		\draw [line width=1.pt] (-0.9942696548396859,-0.5776109478268854)-- (3.8751173205819267,7.870507870313641);
		\draw [line width=1.pt] (-0.9942696548396859,-0.5776109478268854)-- (5.240241307116339,6.356585656658499);
		\draw [line width=1.pt] (-0.9942696548396859,-0.5776109478268854)-- (5.871576485073833,4.418302947484353);
		\draw [line width=1.pt] (-0.9942696548396859,-0.5776109478268854)-- (5.659959266685525,2.390806431953024);
		\draw [line width=1.pt] (-0.9942696548396859,-0.5776109478268854)-- (4.641980191492829,0.6246686768274756);
		\draw [line width=1.pt] (4.641980191492829,0.6246686768274756)-- (5.659959266685525,2.390806431953024);
		\draw [line width=1.pt] (4.641980191492829,0.6246686768274756)-- (5.871576485073833,4.418302947484353);
		\draw [line width=1.pt] (4.641980191492829,0.6246686768274756)-- (5.240241307116339,6.356585656658499);
		\draw [line width=1.pt] (4.641980191492829,0.6246686768274756)-- (3.8751173205819267,7.870507870313641);
		\draw [line width=1.pt] (4.641980191492829,0.6246686768274756)-- (2.0122468645044806,8.698298684157265);
		\draw [line width=1.pt] (4.641980191492829,0.6246686768274756)-- (-0.026262835613412694,8.696825546232356);
		\draw [line width=1.pt] (4.641980191492829,0.6246686768274756)-- (-1.8879349333421296,7.866343175469165);
		\draw [line width=1.pt] (4.641980191492829,0.6246686768274756)-- (-3.2508694102243165,6.350449518468134);
		\draw [line width=1.pt] (4.641980191492829,0.6246686768274756)-- (-3.8794025133360366,4.411256359986279);
		\draw [line width=1.pt] (4.641980191492829,0.6246686768274756)-- (-3.6648551591052874,2.3840678141418907);
		\draw [line width=1.pt] (4.641980191492829,0.6246686768274756)-- (-2.6443245341800923,0.6194031969412541);
		\draw [line width=1.pt] (5.871576485073833,4.418302947484353)-- (5.659959266685525,2.390806431953024);
		\draw [line width=1.pt] (5.871576485073833,4.418302947484353)-- (5.240241307116339,6.356585656658499);
		\draw [line width=1.pt] (5.871576485073833,4.418302947484353)-- (3.8751173205819267,7.870507870313641);
		\draw [line width=1.pt] (5.871576485073833,4.418302947484353)-- (2.0122468645044806,8.698298684157265);
		\draw [line width=1.pt] (5.871576485073833,4.418302947484353)-- (-0.026262835613412694,8.696825546232356);
		\draw [line width=1.pt] (5.871576485073833,4.418302947484353)-- (-1.8879349333421296,7.866343175469165);
		\draw [line width=1.pt] (5.871576485073833,4.418302947484353)-- (-3.2508694102243165,6.350449518468134);
		\draw [line width=1.pt] (5.871576485073833,4.418302947484353)-- (-3.8794025133360366,4.411256359986279);
		\draw [line width=1.pt] (5.871576485073833,4.418302947484353)-- (-3.6648551591052874,2.3840678141418907);
		\draw [line width=1.pt] (5.871576485073833,4.418302947484353)-- (-2.6443245341800923,0.6194031969412541);
		\draw [line width=1.pt] (2.0122468645044806,8.698298684157265)-- (5.240241307116339,6.356585656658499);
		\draw [line width=1.pt] (2.0122468645044806,8.698298684157265)-- (3.8751173205819267,7.870507870313641);
		\draw [line width=1.pt] (2.0122468645044806,8.698298684157265)-- (-0.026262835613412694,8.696825546232356);
		\draw [line width=1.pt] (2.0122468645044806,8.698298684157265)-- (-1.8879349333421296,7.866343175469165);
		\draw [line width=1.pt] (2.0122468645044806,8.698298684157265)-- (-3.2508694102243165,6.350449518468134);
		\draw [line width=1.pt] (2.0122468645044806,8.698298684157265)-- (-3.8794025133360366,4.411256359986279);
		\draw [line width=1.pt] (2.0122468645044806,8.698298684157265)-- (-3.6648551591052874,2.3840678141418907);
		\draw [line width=1.pt] (2.0122468645044806,8.698298684157265)-- (-2.6443245341800923,0.6194031969412541);
		\draw [line width=1.pt] (-0.026262835613412694,8.696825546232356)-- (3.8751173205819267,7.870507870313641);
		\draw [line width=1.pt] (-0.026262835613412694,8.696825546232356)-- (5.240241307116339,6.356585656658499);
		\draw [line width=1.pt] (-0.026262835613412694,8.696825546232356)-- (5.659959266685525,2.390806431953024);
		\draw [line width=1.pt] (2.0122468645044806,8.698298684157265)-- (5.659959266685525,2.390806431953024);
		\draw [line width=1.pt] (-0.026262835613412694,8.696825546232356)-- (-2.6443245341800923,0.6194031969412541);
		\draw [line width=1.pt] (-0.026262835613412694,8.696825546232356)-- (-1.8879349333421296,7.866343175469165);
		\draw [line width=1.pt] (-0.026262835613412694,8.696825546232356)-- (-3.2508694102243165,6.350449518468134);
		\draw [line width=1.pt] (-0.026262835613412694,8.696825546232356)-- (-3.8794025133360366,4.411256359986279);
		\draw [line width=1.pt] (-0.026262835613412694,8.696825546232356)-- (-3.6648551591052874,2.3840678141418907);
		\draw [line width=1.pt] (-3.8794025133360366,4.411256359986279)-- (-3.2508694102243165,6.350449518468134);
		\draw [line width=1.pt] (-3.8794025133360366,4.411256359986279)-- (-1.8879349333421296,7.866343175469165);
		\draw [line width=1.pt] (-3.8794025133360366,4.411256359986279)-- (3.8751173205819267,7.870507870313641);
		\draw [line width=1.pt] (-3.8794025133360366,4.411256359986279)-- (5.240241307116339,6.356585656658499);
		\draw [line width=1.pt] (-3.8794025133360366,4.411256359986279)-- (5.659959266685525,2.390806431953024);
		\draw [line width=1.pt] (-3.8794025133360366,4.411256359986279)-- (-2.6443245341800923,0.6194031969412541);
		\draw [line width=1.pt] (-3.8794025133360366,4.411256359986279)-- (-3.6648551591052874,2.3840678141418907);
		\draw [line width=1.pt] (-2.6443245341800923,0.6194031969412541)-- (5.659959266685525,2.390806431953024);
		\draw [line width=1.pt] (-2.6443245341800923,0.6194031969412541)-- (5.240241307116339,6.356585656658499);
		\draw [line width=1.pt] (-2.6443245341800923,0.6194031969412541)-- (3.8751173205819267,7.870507870313641);
		\draw [line width=1.pt] (-2.6443245341800923,0.6194031969412541)-- (-1.8879349333421296,7.866343175469165);
		\draw [line width=1.pt] (-2.6443245341800923,0.6194031969412541)-- (-3.2508694102243165,6.350449518468134);
		\draw [line width=1.pt] (-2.6443245341800923,0.6194031969412541)-- (-3.6648551591052874,2.3840678141418907);
		\draw [line width=1.pt] (5.659959266685525,2.390806431953024)-- (3.8751173205819267,7.870507870313641);
		\draw [line width=1.pt] (-1.8879349333421296,7.866343175469165)-- (5.659959266685525,2.390806431953024);
		\draw [line width=1.pt] (-3.6648551591052874,2.3840678141418907)-- (5.659959266685525,2.390806431953024);
		\draw [line width=1.pt] (5.240241307116339,6.356585656658499)-- (-3.2508694102243165,6.350449518468134);
		\draw [line width=1.pt] (3.8751173205819267,7.870507870313641)-- (-3.6648551591052874,2.3840678141418907);
		\draw [line width=1.pt] (3.8751173205819267,7.870507870313641)-- (-1.8879349333421296,7.866343175469165);
		\draw [line width=1.pt] (-1.8879349333421296,7.866343175469165)-- (-3.6648551591052874,2.3840678141418907);
		\draw [line width=1.pt] (-0.9942696548396859,-0.5776109478268854)-- (1.,-1.);
		\draw [line width=1.pt] (-2.6443245341800923,0.6194031969412541)-- (1.,-1.);
		\draw [line width=1.pt] (-3.6648551591052874,2.3840678141418907)-- (1.,-1.);
		\draw [line width=1.pt] (-3.8794025133360366,4.411256359986279)-- (1.,-1.);
		\draw [line width=1.pt] (-3.2508694102243165,6.350449518468134)-- (1.,-1.);
		\draw [line width=1.pt] (-1.8879349333421296,7.866343175469165)-- (1.,-1.);
		\draw [line width=1.pt] (-0.026262835613412694,8.696825546232356)-- (1.,-1.);
		\draw [line width=1.pt] (2.0122468645044806,8.698298684157265)-- (1.,-1.);
		\draw [line width=1.pt] (3.8751173205819267,7.870507870313641)-- (1.,-1.);
		\draw [line width=1.pt] (5.240241307116339,6.356585656658499)-- (1.,-1.);
		\draw [line width=1.pt] (5.871576485073833,4.418302947484353)-- (1.,-1.);
		\draw [line width=1.pt] (5.659959266685525,2.390806431953024)-- (1.,-1.);
		\draw [line width=1.pt] (4.641980191492829,0.6246686768274756)-- (1.,-1.);
		\draw [line width=1.pt] (2.993657089646132,-0.5747290551732837)-- (1.,-1.);
		\draw [line width=1.pt] (-4.76303418111067,-0.9849782327152496)-- (1.,-1.);
		\draw [line width=1.pt] (-4.632934599631879,-2.2177177104782277)-- (1.,-1.);
		\draw [line width=1.pt] (-4.206057079291719,-3.471792530391428)-- (1.,-1.);
		\draw [line width=1.pt] (-3.393564195655358,-4.729528426597333)-- (1.,-1.);
		\draw [line width=1.pt] (-2.7330563266424126,-5.390567057711037)-- (1.,-1.);
		\draw [line width=1.pt] (-1.9550595614531452,-5.947768346881919)-- (1.,-1.);
		\draw [line width=1.pt] (6.762864846375749,-0.9533375038290286)-- (1.,-1.);
		\draw [line width=1.pt] (6.626501119770371,-2.247106160423023)-- (1.,-1.);
		\draw [line width=1.pt] (6.191394438175647,-3.502441290704189)-- (1.,-1.);
		\draw [line width=1.pt] (5.404330695240475,-4.716807736903384)-- (1.,-1.);
		\draw [line width=1.pt] (4.72012682257888,-5.401527581428005)-- (1.,-1.);
		\draw [line width=1.pt] (3.891677844983721,-5.985076515658771)-- (1.,-1.);
		\draw [line width=1.pt] (0.9853608434239401,-6.7630351657123375)-- (1.,-1.);
		\draw [line width=1.pt] (-0.28027978423590794,-6.619045497254774)-- (1.,-1.);
		\draw [line width=1.pt] (2.296625560040442,-6.615296143854935)-- (1.,-1.);
		\begin{scriptsize}
			\draw [fill=ffffff] (1.,-1.) circle (2.5pt);
			\draw [fill=magentaIBM] (2.993657089646132,-0.5747290551732837) circle (2.5pt);
			\draw [fill=magentaIBM] (4.641980191492829,0.6246686768274756) circle (2.5pt);
			\draw [fill=celesteIBM] (5.659959266685525,2.390806431953024) circle (2.5pt);
			\draw [fill=magentaIBM] (5.871576485073833,4.418302947484353) circle (2.5pt);
			\draw [fill=gialloIBM] (5.240241307116339,6.356585656658499) circle (2.5pt);
			\draw [fill=celesteIBM] (3.8751173205819267,7.870507870313641) circle (2.5pt);
			\draw [fill=magentaIBM] (2.0122468645044806,8.698298684157265) circle (2.5pt);
			\draw [fill=magentaIBM] (-0.026262835613412694,8.696825546232356) circle (2.5pt);
			\draw [fill=celesteIBM] (-1.8879349333421296,7.866343175469165) circle (2.5pt);
			\draw [fill=gialloIBM] (-3.2508694102243165,6.350449518468134) circle (2.5pt);
			\draw [fill=magentaIBM] (-3.8794025133360366,4.411256359986279) circle (2.5pt);
			\draw [fill=celesteIBM] (-3.6648551591052874,2.3840678141418907) circle (2.5pt);
			\draw [fill=magentaIBM] (-2.6443245341800923,0.6194031969412541) circle (2.5pt);
			\draw [fill=magentaIBM] (-0.9942696548396859,-0.5776109478268854) circle (2.5pt);
			\draw [fill=arancioneIBM] (-4.76303418111067,-0.9849782327152496) circle (2.5pt);
			\draw [fill=arancioneIBM] (-4.632934599631879,-2.2177177104782277) circle (2.5pt);
			\draw [fill=arancioneIBM] (-4.206057079291719,-3.471792530391428) circle (2.5pt);
			\draw [fill=arancioneIBM] (-3.393564195655358,-4.729528426597333) circle (2.5pt);
			\draw [fill=arancioneIBM] (-1.9550595614531452,-5.947768346881919) circle (2.5pt);
			\draw [fill=arancioneIBM] (3.891677844983721,-5.985076515658771) circle (2.5pt);
			\draw [fill=arancioneIBM] (5.404330695240475,-4.716807736903384) circle (2.5pt);
			\draw [fill=arancioneIBM] (6.191394438175647,-3.502441290704189) circle (2.5pt);
			\draw [fill=arancioneIBM] (6.626501119770371,-2.247106160423023) circle (2.5pt);
			\draw [fill=arancioneIBM] (6.762864846375749,-0.9533375038290286) circle (2.5pt);
			\draw [fill=arancioneIBM] (-0.28027978423590794,-6.619045497254774) circle (2.5pt);
			\draw [fill=arancioneIBM] (2.296625560040442,-6.615296143854935) circle (2.5pt);
			\draw [fill=arancioneIBM] (0.9853608434239401,-6.7630351657123375) circle (2.5pt);
			\draw [fill=arancioneIBM] (-2.7330563266424126,-5.390567057711037) circle (2.5pt);
			\draw [fill=arancioneIBM] (4.72012682257888,-5.401527581428005) circle (2.5pt);
		\end{scriptsize}
	\end{tikzpicture}
	
	\caption{$\mathcal{P}(D_{30})$}
	\label{PGD15}
\end{figure}
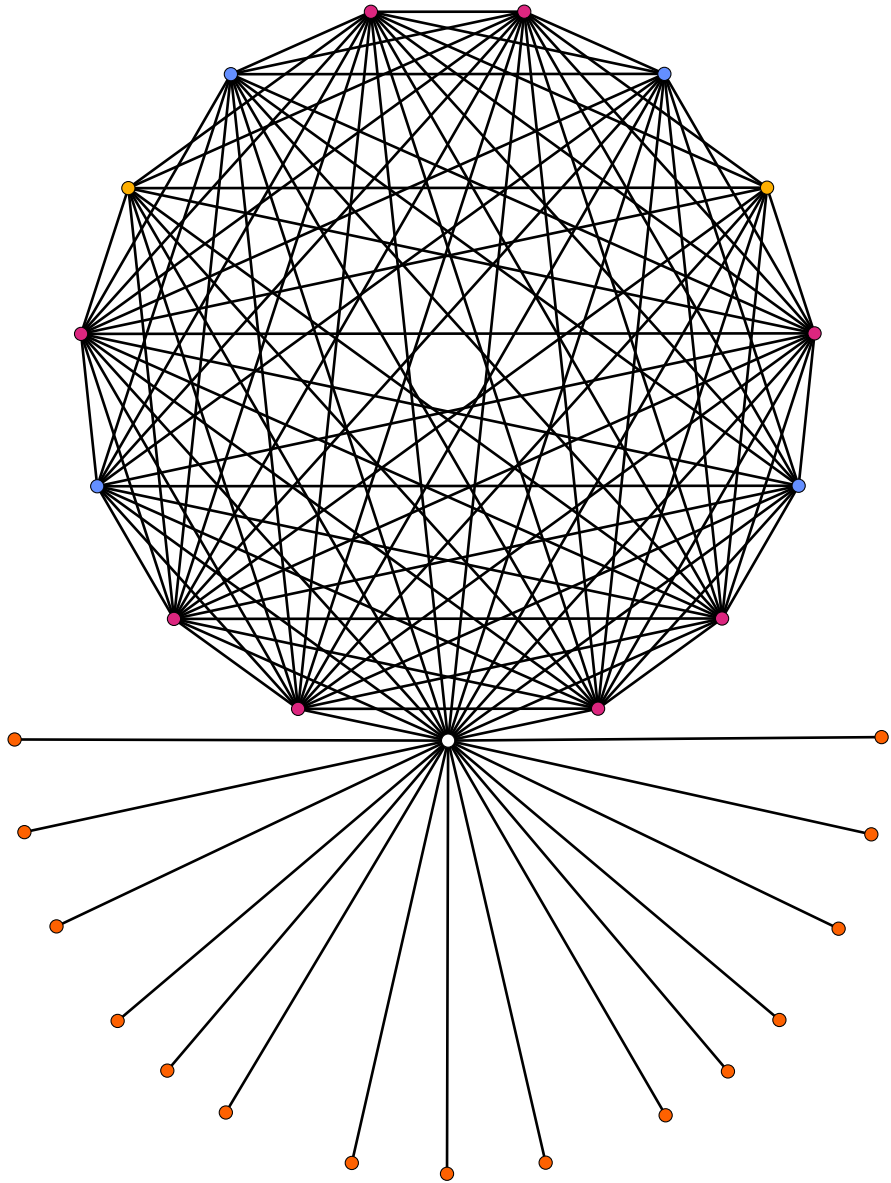

Our friend asks for a rematch. This time she prepares for us the graph $\mathcal{P}(D_{18})$ in Figure \ref{fig:subfig} (a). As before imagine it without colours.
This time a delicate situation arises. We did not face that in the previous case.
The white and orange vertices are, as before, the identity and the involutions.
For the remaining vertices it is easily checked, by the study of the closed neighbourhoods, that they belong to a unique $\mathtt{N}$-class $C$. Since
$\hat{C}=C \cupdot \{1\}$ and $|\hat{C}|= 9$, we have that $C$ is a critical class. Note that we have $|C|=8=3^2-1$ and $|\hat{C}|=9=3^2$ as in the previous case.
Without too much effort it is checked that there exists no $x\in G\setminus \hat{C}$ joined with a vertex of $C$ and hence,
by Proposition \ref{LemmaCriticalClassMio}, $C$ is of compound type.

As a consequence, there exists $y \in C$ with $o(y)=3^2$ such that $ C=\{z\in \langle y \rangle \, | \, 3\leq o(z) \leq 3^2\}.$
Here we face the delicate situation. 
We have seen that knowing the partition in $\diamond$-classes allows us to use Lemma \ref{lemmaClassiDiamond}. But whenever  a compound class appears, we cannot directly see the $\diamond$-partition. However, we can find it up to a graph isomorphism.
We arbitrarily partition the vertices of $C$ in two sets, one formed by $6$ elements and one by $2$ elements. In Figure \ref{fig:subfig} (a) such a partition is revealed by the two colours magenta and blue.
 In $C$ the partition in $\diamond$-classes has exactly two sets with the same sizes as those in our partition. By an argument in the proof of the Main Theorem, our arbitrary partition of the vertices in $C$ is, up to a graph isomorphism, just the
  partition of $C$ into $\diamond$-classes.

We now have in hands, up to an isomorphism, the partition of $G$ into $\diamond$-classes. It is composed by the sets of the partition above and the singletons containing each an involution or $1$. Mimicking the instructions in  Lemma \ref{lemmaClassiDiamond}, we now assign the directions on the edges obtaining the digraph shown in Figure \ref{fig:subfig} (b). By the Main Theorem, that directed graph is just $\vec{\mathcal{P}}(D_{18}).$

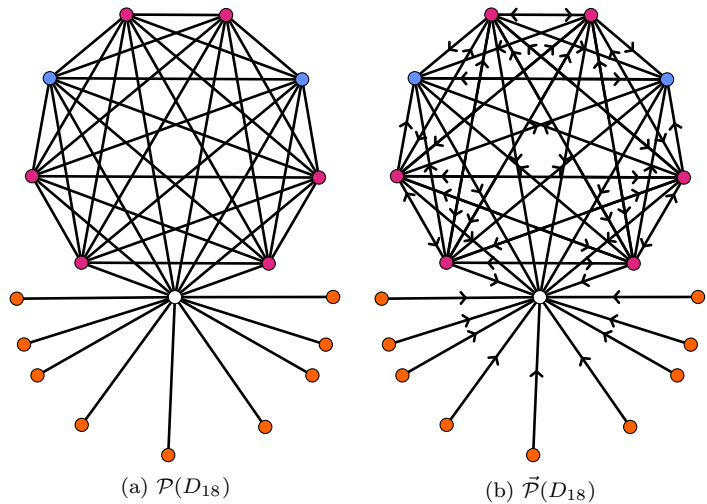
\begin{figure}[h!]
	\centering
	\subfloat[][$\mathcal{P}(D_{18})$]
	{\begin{tikzpicture}[line cap=round,line join=round,>=triangle 45,x=1.0cm,y=1.0cm]
			\clip(1.,-1.2) rectangle (5.5,5.);
			\draw [line width=1.pt] (2.,1.5)-- (3.2427260631821486,1.0409636115652288);
			\draw [line width=1.pt] (1.3430795078613198,2.6504511901978782)-- (3.2427260631821486,1.0409636115652288);
			\draw [line width=1.pt] (1.579344985895828,3.9540082844245044)-- (3.2427260631821486,1.0409636115652288);
			\draw [line width=1.pt] (2.5982451911328757,4.800722430867884)-- (3.2427260631821486,1.0409636115652288);
			\draw [line width=1.pt] (3.92302539372762,4.794406670227618)-- (3.2427260631821486,1.0409636115652288);
			\draw [line width=1.pt] (4.93380621359405,3.938016217099146)-- (3.2427260631821486,1.0409636115652288);
			\draw [line width=1.pt] (5.157632071539149,2.6322656822209947)-- (3.2427260631821486,1.0409636115652288);
			\draw [line width=1.pt] (4.4897723610830855,1.4881302526633837)-- (3.2427260631821486,1.0409636115652288);
			\draw [line width=1.pt] (5.345802988804287,1.0436511761936775)-- (3.2427260631821486,1.0409636115652288);
			\draw [line width=1.pt] (5.2487584059961065,0.4094780474608347)-- (3.2427260631821486,1.0409636115652288);
			\draw [line width=1.pt] (5.070110679918585,0.)-- (3.2427260631821486,1.0409636115652288);
			\draw [line width=1.pt] (4.443418837789942,-0.6856737675889872)-- (3.2427260631821486,1.0409636115652288);
			\draw [line width=1.pt] (3.1534426699475673,-1.0602189730238445)-- (3.2427260631821486,1.0409636115652288);
			\draw [line width=1.pt] (2.0035395443744655,-0.658258725179788)-- (3.2427260631821486,1.0409636115652288);
			\draw [line width=1.pt] (1.415341446445712,0.)-- (3.2427260631821486,1.0409636115652288);
			\draw [line width=1.pt] (1.236070276851942,0.41146195686096987)-- (3.2427260631821486,1.0409636115652288);
			\draw [line width=1.pt] (1.1397454373691325,1.020659288306518)-- (3.2427260631821486,1.0409636115652288);
			\draw [line width=1.pt] (4.4897723610830855,1.4881302526633837)-- (5.157632071539149,2.6322656822209947);
			\draw [line width=1.pt] (4.4897723610830855,1.4881302526633837)-- (4.93380621359405,3.938016217099146);
			\draw [line width=1.pt] (4.4897723610830855,1.4881302526633837)-- (3.92302539372762,4.794406670227618);
			\draw [line width=1.pt] (4.4897723610830855,1.4881302526633837)-- (2.5982451911328757,4.800722430867884);
			\draw [line width=1.pt] (4.4897723610830855,1.4881302526633837)-- (1.579344985895828,3.9540082844245044);
			\draw [line width=1.pt] (4.4897723610830855,1.4881302526633837)-- (1.3430795078613198,2.6504511901978782);
			\draw [line width=1.pt] (4.4897723610830855,1.4881302526633837)-- (2.,1.5);
			\draw [line width=1.pt] (5.157632071539149,2.6322656822209947)-- (1.579344985895828,3.9540082844245044);
			\draw [line width=1.pt] (5.157632071539149,2.6322656822209947)-- (2.,1.5);
			\draw [line width=1.pt] (4.93380621359405,3.938016217099146)-- (1.579344985895828,3.9540082844245044);
			\draw [line width=1.pt] (3.92302539372762,4.794406670227618)-- (2.,1.5);
			\draw [line width=1.pt] (5.157632071539149,2.6322656822209947)-- (3.92302539372762,4.794406670227618);
			\draw [line width=1.pt] (5.157632071539149,2.6322656822209947)-- (4.93380621359405,3.938016217099146);
			\draw [line width=1.pt] (5.157632071539149,2.6322656822209947)-- (2.5982451911328757,4.800722430867884);
			\draw [line width=1.pt] (5.157632071539149,2.6322656822209947)-- (1.3430795078613198,2.6504511901978782);
			\draw [line width=1.pt] (3.92302539372762,4.794406670227618)-- (4.93380621359405,3.938016217099146);
			\draw [line width=1.pt] (3.92302539372762,4.794406670227618)-- (1.3430795078613198,2.6504511901978782);
			\draw [line width=1.pt] (3.92302539372762,4.794406670227618)-- (1.579344985895828,3.9540082844245044);
			\draw [line width=1.pt] (3.92302539372762,4.794406670227618)-- (2.5982451911328757,4.800722430867884);
			\draw [line width=1.pt] (2.5982451911328757,4.800722430867884)-- (1.579344985895828,3.9540082844245044);
			\draw [line width=1.pt] (2.5982451911328757,4.800722430867884)-- (1.3430795078613198,2.6504511901978782);
			\draw [line width=1.pt] (2.5982451911328757,4.800722430867884)-- (2.,1.5);
			\draw [line width=1.pt] (2.5982451911328757,4.800722430867884)-- (4.93380621359405,3.938016217099146);
			\draw [line width=1.pt] (1.579344985895828,3.9540082844245044)-- (1.3430795078613198,2.6504511901978782);
			\draw [line width=1.pt] (1.579344985895828,3.9540082844245044)-- (2.,1.5);
			\draw [line width=1.pt] (1.3430795078613198,2.6504511901978782)-- (2.,1.5);
			\draw [line width=1.pt] (1.3430795078613198,2.6504511901978782)-- (4.93380621359405,3.938016217099146);
			\draw [line width=1.pt] (4.93380621359405,3.938016217099146)-- (2.,1.5);
			\begin{scriptsize}
				\draw [fill=magentaIBM] (2.,1.5) circle (2.5pt);
				\draw [fill=ffffff] (3.2427260631821486,1.0409636115652288) circle (2.5pt);
				\draw [fill=magentaIBM] (4.4897723610830855,1.4881302526633837) circle (2.5pt);
				\draw [fill=magentaIBM] (5.157632071539149,2.6322656822209947) circle (2.5pt);
				\draw [fill=celesteIBM] (4.93380621359405,3.938016217099146) circle (2.5pt);
				\draw [fill=magentaIBM] (3.92302539372762,4.794406670227618) circle (2.5pt);
				\draw [fill=magentaIBM] (2.5982451911328757,4.800722430867884) circle (2.5pt);
				\draw [fill=celesteIBM] (1.579344985895828,3.9540082844245044) circle (2.5pt);
				\draw [fill=magentaIBM] (1.3430795078613198,2.6504511901978782) circle (2.5pt);
				\draw [fill=arancioneIBM] (1.1397454373691325,1.020659288306518) circle (2.5pt);
				\draw [fill=arancioneIBM] (5.345802988804287,1.0436511761936775) circle (2.5pt);
				\draw [fill=arancioneIBM] (1.415341446445712,0.) circle (2.5pt);
				\draw [fill=arancioneIBM] (5.070110679918585,0.) circle (2.5pt);
				\draw [fill=arancioneIBM] (1.236070276851942,0.41146195686096987) circle (2.5pt);
				\draw [fill=arancioneIBM] (5.2487584059961065,0.4094780474608347) circle (2.5pt);
				\draw [fill=arancioneIBM] (2.0035395443744655,-0.658258725179788) circle (2.5pt);
				\draw [fill=arancioneIBM] (4.443418837789942,-0.6856737675889872) circle (2.5pt);
				\draw [fill=arancioneIBM] (3.1534426699475673,-1.0602189730238445) circle (2.5pt);
			\end{scriptsize}
	\end{tikzpicture}} \quad
	\subfloat[][$\vec{\mathcal{P}}(D_{18})$]
	{\begin{tikzpicture}[line cap=round,line join=round,>=triangle 45,x=1.0cm,y=1.0cm]
			\clip(1.,-1.2) rectangle (5.5,5.);
			\draw [line width=1.pt] (2.,1.5)-- (3.2427260631821486,1.0409636115652288);
			\draw [line width=1.pt] (2.681086097180521,1.2484213846273409) -- (2.594890526204747,1.1988141270752775);
			\draw [line width=1.pt] (2.681086097180521,1.2484213846273409) -- (2.647835536977402,1.3421494844899506);
			\draw [line width=1.pt] (1.3430795078613198,2.6504511901978782)-- (3.2427260631821486,1.0409636115652288);
			\draw [line width=1.pt] (2.3414790011414994,1.8045508944577628) -- (2.2435149778131858,1.787415942137835);
			\draw [line width=1.pt] (2.3414790011414994,1.8045508944577628) -- (2.3422905932302824,1.903998859625271);
			\draw [line width=1.pt] (1.579344985895828,3.9540082844245044)-- (3.2427260631821486,1.0409636115652288);
			\draw [line width=1.pt] (2.4426058816767378,2.442197438382612) -- (2.3446893130042827,2.4596015194295666);
			\draw [line width=1.pt] (2.4426058816767378,2.442197438382612) -- (2.4773817360736934,2.5353703765601665);
			\draw [line width=1.pt] (3.92302539372762,4.794406670227618)-- (3.2427260631821486,1.0409636115652288);
			\draw [line width=1.pt] (3.571521256170747,2.8550386509381007) -- (3.507699940504898,2.931310507637388);
			\draw [line width=1.pt] (3.571521256170747,2.8550386509381007) -- (3.658051516404871,2.904059774155458);
			\draw [line width=1.pt] (4.93380621359405,3.938016217099146)-- (3.2427260631821486,1.0409636115652288);
			\draw [line width=1.pt] (4.056170063095382,2.434504928238221) -- (4.022284155075341,2.5280052046834482);
			\draw [line width=1.pt] (4.056170063095382,2.434504928238221) -- (4.154248121700858,2.4509746239809256);
			\draw [line width=1.pt] (5.157632071539149,2.6322656822209947)-- (3.2427260631821486,1.0409636115652288);
			\draw [line width=1.pt] (4.151212649427276,1.795923165478403) -- (4.151349289663,1.89537434841316);
			\draw [line width=1.pt] (4.151212649427276,1.795923165478403) -- (4.2490088450583,1.7778549453730632);
			\draw [line width=1.pt] (4.4897723610830855,1.4881302526633837)-- (3.2427260631821486,1.0409636115652288);
			\draw [line width=1.pt] (3.8063185242287143,1.243056948579774) -- (3.8404612318911786,1.3364637575989888);
			\draw [line width=1.pt] (3.8063185242287143,1.243056948579774) -- (3.892037192374054,1.1926301066296228);
			\draw [line width=1.pt] (5.345802988804287,1.0436511761936775)-- (3.2427260631821486,1.0409636115652288);
			\draw [line width=1.pt] (4.230597422044172,1.0422260323969108) -- (4.294166892214167,1.118707918618308);
			\draw [line width=1.pt] (4.230597422044172,1.0422260323969108) -- (4.294362159772268,0.9659068691405985);
			\draw [line width=1.pt] (5.2487584059961065,0.40947804746083494)-- (3.2427260631821486,1.0409636115652288);
			\draw [line width=1.pt] (4.185012991574253,0.7443379890245708) -- (4.268682826002975,0.7980959211308819);
			\draw [line width=1.pt] (4.185012991574253,0.7443379890245708) -- (4.2228016431752815,0.6523457378951816);
			\draw [line width=1.pt] (5.070110679918585,0.)-- (3.2427260631821486,1.0409636115652288);
			\draw [line width=1.pt] (4.1010973868868,0.5519952219859193) -- (4.194234470994333,0.5868669873788962);
			\draw [line width=1.pt] (4.1010973868868,0.5519952219859193) -- (4.1186022721064015,0.45409662418633323);
			\draw [line width=1.pt] (4.443418837789942,-0.6856737675889872)-- (3.2427260631821486,1.0409636115652288);
			\draw [line width=1.pt] (3.8067235034929596,0.2299159540942808) -- (3.905797689013438,0.22126365837982462);
			\draw [line width=1.pt] (3.8067235034929596,0.2299159540942808) -- (3.7803472119586545,0.1340261855964171);
			\draw [line width=1.pt] (3.1534426699475673,-1.0602189730238445)-- (3.2427260631821486,1.0409636115652288);
			\draw [line width=1.pt] (3.200787270585503,0.05398207524259128) -- (3.274416073731137,-0.012871165554081913);
			\draw [line width=1.pt] (3.200787270585503,0.05398207524259128) -- (3.1217526593985787,-0.0063841959045346855);
			\draw [line width=1.pt] (2.0035395443744655,-0.658258725179788)-- (3.2427260631821486,1.0409636115652288);
			\draw [line width=1.pt] (2.660647083899965,0.24279353116627567) -- (2.6848621093465748,0.14633530704672912);
			\draw [line width=1.pt] (2.660647083899965,0.24279353116627567) -- (2.5614034982100393,0.23636957933870917);
			\draw [line width=1.pt] (1.415341446445712,0.)-- (3.2427260631821486,1.0409636115652288);
			\draw [line width=1.pt] (2.3843547394774975,0.5519952219859193) -- (2.3668498542578957,0.45409662418633323);
			\draw [line width=1.pt] (2.3843547394774975,0.5519952219859193) -- (2.2912176553699646,0.5868669873788962);
			\draw [line width=1.pt] (1.236070276851942,0.41146195686096987)-- (3.2427260631821486,1.0409636115652288);
			\draw [line width=1.pt] (2.300146286732008,0.74526988421637) -- (2.26226669002097,0.6533150441551434);
			\draw [line width=1.pt] (2.300146286732008,0.74526988421637) -- (2.2165296500131206,0.7991105242710549);
			\draw [line width=1.pt] (1.1397454373691325,1.020659288306518)-- (3.2427260631821486,1.0409636115652288);
			\draw [line width=1.pt] (2.254899938910518,1.0314261290490232) -- (2.191973365211421,0.9544144235740202);
			\draw [line width=1.pt] (2.254899938910518,1.0314261290490232) -- (2.190498135339861,1.1072084762977266);
			\draw [line width=1.pt] (4.4897723610830855,1.4881302526633837)-- (4.93380621359405,3.938016217099146);
			\draw [line width=1.pt] (4.7231437596227055,2.775719724839587) -- (4.786965075288554,2.6994478681402994);
			\draw [line width=1.pt] (4.7231437596227055,2.775719724839587) -- (4.636613499388582,2.72669860162223);
			\draw [line width=1.pt] (4.4897723610830855,1.4881302526633837)-- (3.92302539372762,4.794406670227618);
			\draw [line width=1.pt] (4.4897723610830855,1.4881302526633837)-- (2.5982451911328757,4.800722430867884);
			\draw [line width=1.pt] (4.4897723610830855,1.4881302526633837)-- (1.579344985895828,3.9540082844245044);
			\draw [line width=1.pt] (2.9859824578696914,2.762225774967734) -- (3.083946481198005,2.7793607272876617);
			\draw [line width=1.pt] (2.9859824578696914,2.762225774967734) -- (2.9851708657809084,2.6627778098002257);
			\draw [line width=1.pt] (4.4897723610830855,1.4881302526633837)-- (1.3430795078613198,2.6504511901978782);
			\draw [line width=1.pt] (5.157632071539149,2.6322656822209947)-- (1.579344985895828,3.9540082844245044);
			\draw [line width=1.pt] (3.3087654631280423,3.3151974044780226) -- (3.3949610341038166,3.3648046620300858);
			\draw [line width=1.pt] (3.3087654631280423,3.3151974044780226) -- (3.3420160233311615,3.221469304615413);
			\draw [line width=1.pt] (5.157632071539149,2.6322656822209947)-- (2.,1.5);
			\draw [line width=1.pt] (4.93380621359405,3.938016217099146)-- (1.579344985895828,3.9540082844245044);
			\draw [line width=1.pt] (3.92302539372762,4.794406670227618)-- (2.,1.5);
			\draw [line width=1.pt] (5.157632071539149,2.6322656822209947)-- (4.93380621359405,3.938016217099146);
			\draw [line width=1.pt] (5.034962494910526,3.3478928522300517) -- (5.121021425650578,3.2980489268473576);
			\draw [line width=1.pt] (5.034962494910526,3.3478928522300517) -- (4.970416859482621,3.272232972472782);
			\draw [line width=1.pt] (5.157632071539149,2.6322656822209947)-- (2.5982451911328757,4.800722430867884);
			\draw [line width=1.pt] (5.157632071539149,2.6322656822209947)-- (1.3430795078613198,2.6504511901978782);
			\draw [line width=1.pt] (3.92302539372762,4.794406670227618)-- (4.93380621359405,3.938016217099146);
			\draw [line width=1.pt] (4.476992019280601,4.3250549372395914) -- (4.3790279959522875,4.307919984919663);
			\draw [line width=1.pt] (4.476992019280601,4.3250549372395914) -- (4.477803611369384,4.4245029024071005);
			\draw [line width=1.pt] (3.92302539372762,4.794406670227618)-- (1.3430795078613198,2.6504511901978782);
			\draw [line width=1.pt] (3.92302539372762,4.794406670227618)-- (1.579344985895828,3.9540082844245044);
			\draw [line width=1.pt] (2.691254501907822,4.35271749379153) -- (2.725397209570286,4.446124302810744);
			\draw [line width=1.pt] (2.691254501907822,4.35271749379153) -- (2.776973170053162,4.3022906518413775);
			\draw [line width=1.pt] (2.5982451911328757,4.800722430867884)-- (1.579344985895828,3.9540082844245044);
			\draw [line width=1.pt] (2.0398286705809783,4.336673876231486) -- (2.0399653108167013,4.436125059166242);
			\draw [line width=1.pt] (2.0398286705809783,4.336673876231486) -- (2.1376248662120023,4.318605656126145);
			\draw [line width=1.pt] (2.5982451911328757,4.800722430867884)-- (2.,1.5);
			\draw [line width=1.pt] (2.5982451911328757,4.800722430867884)-- (4.93380621359405,3.938016217099146);
			\draw [line width=1.pt] (3.8257487679529096,4.3473089028282415) -- (3.7395531969771354,4.297701645276178);
			\draw [line width=1.pt] (3.8257487679529096,4.3473089028282415) -- (3.792498207749791,4.441037002690851);
			\draw [line width=1.pt] (1.3430795078613198,2.6504511901978782)-- (4.93380621359405,3.938016217099146);
			\draw [line width=1.pt] (3.1983735486315874,3.3157236871830436) -- (3.164230840969123,3.222316878163829);
			\draw [line width=1.pt] (3.1983735486315874,3.3157236871830436) -- (3.112654880486247,3.366150529133195);
			\draw [line width=1.pt] (4.823702216311117,2.060197967442189)-- (5.157632071539149,2.6322656822209947);
			\draw [line width=1.pt] (5.022763219217851,2.401216810925557) -- (5.056649127237892,2.3077165344803303);
			\draw [line width=1.pt] (5.022763219217851,2.401216810925557) -- (4.924685160612375,2.3847471151828525);
			\draw [line width=1.pt] (4.823702216311117,2.060197967442189)-- (4.4897723610830855,1.4881302526633837);
			\draw [line width=1.pt] (4.624641213404384,1.7191791239588206) -- (4.590755305384343,1.8126794004040474);
			\draw [line width=1.pt] (4.624641213404384,1.7191791239588206) -- (4.7227192720098605,1.7356488197015256);
			\draw [line width=1.pt] (3.2448861805415428,1.4940651263316918)-- (2.,1.5);
			\draw [line width=1.pt] (2.558776657840304,1.4973360866841345) -- (2.6228073184927174,1.573432282082407);
			\draw [line width=1.pt] (2.558776657840304,1.4973360866841345) -- (2.622078862048825,1.4206328442492855);
			\draw [line width=1.pt] (3.2448861805415428,1.4940651263316918)-- (4.4897723610830855,1.4881302526633837);
			\draw [line width=1.pt] (3.930995703242781,1.4907941659792487) -- (3.866965042590368,1.4146979705809763);
			\draw [line width=1.pt] (3.930995703242781,1.4907941659792487) -- (3.86769349903426,1.5674974084140978);
			\draw [line width=1.pt] (4.540328732633385,3.7133361762243062)-- (3.92302539372762,4.794406670227618);
			\draw [line width=1.pt] (4.200106706042753,4.309159932838216) -- (4.298023274715208,4.291755851791261);
			\draw [line width=1.pt] (4.200106706042753,4.309159932838216) -- (4.165330851645797,4.215986994660661);
			\draw [line width=1.pt] (4.540328732633385,3.7133361762243062)-- (5.157632071539149,2.6322656822209947);
			\draw [line width=1.pt] (4.880550759224017,3.117512419610396) -- (4.782634190551563,3.1349165006573503);
			\draw [line width=1.pt] (4.880550759224017,3.117512419610396) -- (4.915326613620972,3.2106853577879497);
			\draw [line width=1.pt] (3.2606352924302477,4.797564550547751)-- (2.5982451911328757,4.800722430867884);
			\draw [line width=1.pt] (2.8657738093510945,4.799447014226105) -- (2.929804470003508,4.875543209624378);
			\draw [line width=1.pt] (2.8657738093510945,4.799447014226105) -- (2.9290760135596154,4.722743771791257);
			\draw [line width=1.pt] (3.2606352924302477,4.797564550547751)-- (3.92302539372762,4.794406670227618);
			\draw [line width=1.pt] (3.655496775509401,4.795682086869396) -- (3.591466114856988,4.719585891471124);
			\draw [line width=1.pt] (3.655496775509401,4.795682086869396) -- (3.59219457130088,4.872385329304245);
			\draw [line width=1.pt] (1.9851289537725336,3.750370090970045)-- (2.5982451911328757,4.800722430867884);
			\draw [line width=1.pt] (2.3237831477454223,4.330531247012931) -- (2.3576690557654634,4.237030970567703);
			\draw [line width=1.pt] (2.3237831477454223,4.330531247012931) -- (2.225705089139946,4.314061551270226);
			\draw [line width=1.pt] (1.9851289537725336,3.750370090970045)-- (1.3430795078613198,2.6504511901978782);
			\draw [line width=1.pt] (1.6320081555242092,3.145425654489996) -- (1.5981222475041676,3.2389259309352227);
			\draw [line width=1.pt] (1.6320081555242092,3.145425654489996) -- (1.7300862141296856,3.1618953502327005);
			\draw [line width=1.pt] (1.3430795078613198,2.6504511901978782)-- (1.579344985895828,3.9540082844245044);
			\draw [line width=1.pt] (1.4725667191627114,3.364876227269513) -- (1.5363880348285606,3.288604370570226);
			\draw [line width=1.pt] (1.4725667191627114,3.364876227269513) -- (1.3860364589285874,3.3158551040521558);
			\draw [line width=1.pt] (2.349529699131545,3.428474370171321)-- (2.5982451911328757,4.800722430867884);
			\draw [line width=1.pt] (2.485241917416348,4.177244890477925) -- (2.549063233082197,4.100973033778637);
			\draw [line width=1.pt] (2.485241917416348,4.177244890477925) -- (2.3987116571822242,4.128223767260567);
			\draw [line width=1.pt] (2.252981197226564,2.895783408960269)-- (2.,1.5);
			\draw [line width=1.pt] (2.1151361263291446,2.135245214521812) -- (2.0513148106632952,2.2115170712211);
			\draw [line width=1.pt] (2.1151361263291446,2.135245214521812) -- (2.2016663865632684,2.1842663377391696);
			\draw [line width=1.pt] (4.158722734895537,3.4194005287350056)-- (3.92302539372762,4.794406670227618);
			\draw [line width=1.pt] (4.030117416655505,4.1696555020512935) -- (4.116176347395557,4.119811576668599);
			\draw [line width=1.pt] (4.030117416655505,4.1696555020512935) -- (3.9655717812276,4.093995622294024);
			\draw [line width=1.pt] (4.238691937599109,2.9528778929563515)-- (4.4897723610830855,1.4881302526633837);
			\draw [line width=1.pt] (4.3749887969971715,2.1577521702398856) -- (4.28892986625712,2.2075960956225793);
			\draw [line width=1.pt] (4.3749887969971715,2.1577521702398856) -- (4.439534432425076,2.2334120499971557);
			\draw [line width=1.pt] (3.693011744500132,3.873174532964584)-- (2.5982451911328757,4.800722430867884);
			\draw [line width=1.pt] (3.097052252196739,4.378104988340024) -- (3.1950162755250524,4.395239940659952);
			\draw [line width=1.pt] (3.097052252196739,4.378104988340024) -- (3.0962406601079553,4.278657023172516);
			\draw [line width=1.pt] (4.112095555875139,3.518103118798182)-- (5.157632071539149,2.6322656822209947);
			\draw [line width=1.pt] (4.68344002932691,3.034027894085798) -- (4.585476005998597,3.01689294176587);
			\draw [line width=1.pt] (4.68344002932691,3.034027894085798) -- (4.684251621415694,3.133475859253306);
			\draw [line width=1.pt] (2.9494006628168052,3.9474766773682366)-- (1.579344985895828,3.9540082844245044);
			\draw [line width=1.pt] (2.2007063919258494,3.9510460044146587) -- (2.264737052578263,4.02714219981293);
			\draw [line width=1.pt] (2.2007063919258494,3.9510460044146587) -- (2.2640085961343703,3.8743427619798094);
			\draw [line width=1.pt] (3.5431283140845564,3.9446461385939027)-- (4.93380621359405,3.938016217099146);
			\draw [line width=1.pt] (4.302133696269771,3.9410276543282357) -- (4.238103035617358,3.8649314589299633);
			\draw [line width=1.pt] (4.302133696269771,3.9410276543282357) -- (4.238831492061251,4.0177308967630845);
			\draw [line width=1.pt] (2.,1.5)-- (4.93380621359405,3.938016217099146);
			\draw [line width=1.pt] (3.515869524730399,2.7596995899642813) -- (3.515732884494676,2.660248407029524);
			\draw [line width=1.pt] (3.515869524730399,2.7596995899642813) -- (3.418073329099375,2.777767810069621);
			\draw [line width=1.pt] (2.,1.5)-- (1.579344985895828,3.9540082844245044);
			\draw [line width=1.pt] (1.7789158452918408,2.789756044782234) -- (1.864974776031892,2.7399121193995404);
			\draw [line width=1.pt] (1.7789158452918408,2.789756044782234) -- (1.714370209863936,2.714096165024964);
			\draw [line width=1.pt] (1.6715397539306598,2.075225595098939)-- (1.3430795078613198,2.6504511901978782);
			\draw [line width=1.pt] (1.47573927375824,2.4181269022606626) -- (1.573655842430695,2.400722821213708);
			\draw [line width=1.pt] (1.47573927375824,2.4181269022606626) -- (1.4409634193612846,2.3249539640831087);
			\draw [line width=1.pt] (1.6715397539306598,2.075225595098939)-- (2.,1.5);
			\draw [line width=1.pt] (1.8673402341030796,1.732324287937215) -- (1.7694236654306248,1.7497283689841694);
			\draw [line width=1.pt] (1.8673402341030796,1.732324287937215) -- (1.9021160885000352,1.825497226114769);
			\draw [line width=1.pt] (2.5982451911328757,4.800722430867884)-- (3.2427260631821486,1.0409636115652288);
			\draw [line width=1.pt] (2.9312422748135862,2.8580911186465743) -- (2.845183344073535,2.9079350440292684);
			\draw [line width=1.pt] (2.9312422748135862,2.8580911186465743) -- (2.9957879102414906,2.9337509984038443);
			\draw [line width=1.pt] (3.2342219060273583,3.686949642008015)-- (2.5982451911328757,4.800722430867884);
			\draw [line width=1.pt] (2.884663191442367,4.299124546050203) -- (2.982579760114822,4.281720465003249);
			\draw [line width=1.pt] (2.884663191442367,4.299124546050203) -- (2.8498873370454114,4.20595160787265);
			\draw [line width=1.pt] (3.763113165245763,2.760713383686791)-- (4.4897723610830855,1.4881302526633837);
			\draw [line width=1.pt] (4.158013120302174,2.0691333085628334) -- (4.060096551629719,2.086537389609788);
			\draw [line width=1.pt] (4.158013120302174,2.0691333085628334) -- (4.192788974699129,2.1623062467403873);
			\draw [line width=1.pt] (2.7155668451375723,2.1434837666109763)-- (1.3430795078613198,2.6504511901978782);
			\draw [line width=1.pt] (1.9696001109099994,2.4190278995597003) -- (2.055795681885774,2.4686351571117635);
			\draw [line width=1.pt] (1.9696001109099994,2.4190278995597003) -- (2.002850671113119,2.3252997996970906);
			\draw [line width=1.pt] (3.1666483659755613,1.9768639150204133)-- (4.4897723610830855,1.4881302526633837);
			\draw [line width=1.pt] (3.88793342911877,1.7104366626866254) -- (3.801737858142996,1.660829405134562);
			\draw [line width=1.pt] (3.88793342911877,1.7104366626866254) -- (3.8546828689156514,1.804164762549235);
			\draw [line width=1.pt] (3.4014536698064948,2.002534133044542)-- (2.,1.5);
			\draw [line width=1.pt] (2.640796146999345,1.729777082987739) -- (2.67493885466181,1.8231838920069539);
			\draw [line width=1.pt] (2.640796146999345,1.729777082987739) -- (2.726514815144686,1.6793502410375878);
			\draw [line width=1.pt] (3.839160339799589,2.159487264405042)-- (5.157632071539149,2.6322656822209947);
			\draw [line width=1.pt] (4.558326893573272,2.4173664568475495) -- (4.5241841859108085,2.323959647828335);
			\draw [line width=1.pt] (4.558326893573272,2.4173664568475495) -- (4.4726082254279325,2.4677932987977007);
			\draw [line width=1.pt] (3.2395327542224077,3.6234898856745454)-- (3.92302539372762,4.794406670227618);
			\draw [line width=1.pt] (3.613375149267731,4.263933264045047) -- (3.6472610572877726,4.1704329875998205);
			\draw [line width=1.pt] (3.613375149267731,4.263933264045047) -- (3.515297090662255,4.247463568302342);
			\draw [line width=1.pt] (2.7026046458259447,2.7036582768442416)-- (2.,1.5);
			\draw [line width=1.pt] (2.319206247620255,2.0468441523281546) -- (2.2853203396002133,2.140344428773382);
			\draw [line width=1.pt] (2.319206247620255,2.0468441523281546) -- (2.4172843062257314,2.0633138480708593);
			\draw [line width=1.pt] (2.677287141297686,2.644090485198893)-- (1.3430795078613198,2.6504511901978782);
			\draw [line width=1.pt] (1.946516892149036,2.6475743612166736) -- (2.0105475528014494,2.723670556614946);
			\draw [line width=1.pt] (1.946516892149036,2.6475743612166736) -- (2.009819096357557,2.5708711187818247);
			\draw [line width=1.pt] (3.806211441031346,2.638708448507725)-- (5.157632071539149,2.6322656822209947);
			\draw [line width=1.pt] (4.5455881887157155,2.635183541846071) -- (4.481557528063303,2.559087346447799);
			\draw [line width=1.pt] (4.5455881887157155,2.635183541846071) -- (4.482285984507195,2.71188678428092);
			\draw [line width=1.pt] (2.4606071000121705,3.579125476895851)-- (1.3430795078613198,2.6504511901978782);
			\draw [line width=1.pt] (1.8528768860033715,3.074096852132156) -- (1.8530135262390945,3.1735480350669127);
			\draw [line width=1.pt] (1.8528768860033715,3.074096852132156) -- (1.9506730816343956,3.056028632026816);
			\draw [line width=1.pt] (2.8051024467994927,3.8654038405761124)-- (3.92302539372762,4.794406670227618);
			\draw [line width=1.pt] (3.41303033819693,4.370596736816574) -- (3.412893697961207,4.271145553881817);
			\draw [line width=1.pt] (3.41303033819693,4.370596736816574) -- (3.315234142565906,4.388664956921914);
			\begin{scriptsize}
				\draw [fill=magentaIBM] (2.,1.5) circle (2.5pt);
				\draw [fill=ffffff] (3.2427260631821486,1.0409636115652288) circle (2.5pt);
				\draw [fill=magentaIBM] (4.4897723610830855,1.4881302526633837) circle (2.5pt);
				\draw [fill=magentaIBM] (5.157632071539149,2.6322656822209947) circle (2.5pt);
				\draw [fill=celesteIBM] (4.93380621359405,3.938016217099146) circle (2.5pt);
				\draw [fill=magentaIBM] (3.92302539372762,4.794406670227618) circle (2.5pt);
				\draw [fill=magentaIBM] (2.5982451911328757,4.800722430867884) circle (2.5pt);
				\draw [fill=celesteIBM] (1.579344985895828,3.9540082844245044) circle (2.5pt);
				\draw [fill=magentaIBM] (1.3430795078613198,2.6504511901978782) circle (2.5pt);
				\draw [fill=arancioneIBM] (1.1397454373691325,1.020659288306518) circle (2.5pt);
				\draw [fill=arancioneIBM] (5.345802988804287,1.0436511761936775) circle (2.5pt);
				\draw [fill=arancioneIBM] (1.415341446445712,0.) circle (2.5pt);
				\draw [fill=arancioneIBM] (5.070110679918585,0.) circle (2.5pt);
				\draw [fill=arancioneIBM] (1.236070276851942,0.41146195686096987) circle (2.5pt);
				\draw [fill=arancioneIBM] (5.2487584059961065,0.40947804746083494) circle (2.5pt);
				\draw [fill=arancioneIBM] (2.0035395443744655,-0.658258725179788) circle (2.5pt);
				\draw [fill=arancioneIBM] (4.443418837789942,-0.6856737675889872) circle (2.5pt);
				\draw [fill=arancioneIBM] (3.1534426699475673,-1.0602189730238445) circle (2.5pt);
			\end{scriptsize}
	\end{tikzpicture}} \\
	\caption{Example of reconstruction of the directed power graph from a power graph: elements of the same order have the same colour.}
	\label{fig:subfig}
\end{figure}
	
		\vspace{4mm}
\noindent {{\bf Acknowledgments}} We wish to thank an anonymous referee for the care in dealing with our paper and the many suggestions for improving clarity and readability. We also wish to thank Peter Cameron for his encouragement to deepen the topics of  his paper \cite{Cameron_2}, 
Dikran Dikranjan, José C\'aceres and Josef \v{S}lapal  for illuminating conversations about the Moore closure operators for graphs.  Daniela Bubboloni is partially supported by GNSAGA of INdAM (Italy), and local funding from the Universit\`a degli Studi di Firenze. This work is also funded by the European Union - Next Generation EU, Missione 4 Componente 1, CUP B53D23009410006, PRIN 2022- 2022PSTWLB - Group Theory and Applications.

		\vspace{10mm}
\noindent {\Large{\bf Conflict of interest}}
\vspace{2mm}

\noindent Declarations of conflict of interest in the manuscript: none.

\section{Appendix}

\subsection{The algorithm for the reconstruction}\label{sezione-alg}

The proof of the Main Theorem  is a constructive proof that can be explicitly converted into an algorithm. 
We call this algorithm Reconstruction Algorithm of Directed Power Graphs (see Algorithm \ref{ReconstructionOfDirectedPG}).
We now briefly analyze and comment the main steps of the Reconstruction Algorithm of Directed Power Graphs. Such algorithm is divided into three sub-algorithms (see Algorithms \ref{ClassTypeDistinctionAlgorithm}, \ref{CompoundClassPartitioningAlgorithm}, \ref{ArcsDirectionsChoiceAlgorithm}). 
The blue texts appearing in the pseudo-code represent some helpful comments.

The input of the Reconstruction Algorithm of Directed Power Graphs is a graph known to be the power graph of a finite group $G$. We think $G$ inside a black box with no access.
In other words, the group $G$ is unknown and, usually, it will remain undisclosed till the end of the algorithm. This reflects the fact non-isomorphic groups may have isomorphic power graphs.
Note also that, in order to deal with the graph vertices, we need to label them. A different labelling may give a different output, but the digraphs obtained are isomorphic.

We start evaluating the cardinality $|\mathcal{S}|$ of the star class.
Such information breaks the algorithm into two possible roads to follow.
If $|S|>1$, using  Proposition \ref{S>1}, we recognize the group in play and hence we trivially  reconstruct the directed power graph. 
The case of true interest is then when $|\mathcal{S}|=1$.
Since the star class has only one element, that element is necessarily the identity $1\in G$.
We partition the vertex set in $\mathtt{N}$-classes.
Then, for each $\mathtt{N}$-class $C \ne \mathcal{S}$, we determine the type with Algorithm \ref{ClassTypeDistinctionAlgorithm}. In this part we use the information from Proposition \ref{CarachetisationN-classes_1}, \ref{CarachetisationN-classes_2} and \ref{LemmaCriticalClassMio} in order to distinguish the classes of plain type from the ones of compound type.

Knowing the type of every $\mathtt{N}$-class allows us to create a partition of the vertex set that, up to a graph isomorphism, is the one in $\diamond$-classes using Algorithm \ref{CompoundClassPartitioningAlgorithm}. In this step we only have to partition all the compound classes in subset of given cardinalities.

Once we have the partition of the vertex set in $\diamond$-classes, from the  knowledge of the edge set of the graph we reconstruct the directed power graph as shown in Algorithm \ref{ArcsDirectionsChoiceAlgorithm}. Lemma \ref{lemmaClassiDiamond} is the main player here.

\begin{algorithm}
	\caption{ClassTypeDistinctionAlgorithm}\label{ClassTypeDistinctionAlgorithm}
	\begin{algorithmic}[1]
		\Require A power graph $\Gamma=(V, E)$, $C\ne \mathcal{S}$ an $\mathtt{N}$-class of $\Gamma$
		\Ensure The type of $C$ 
		\State $\hat{C} \gets N[N[C]]$
		\State $\hat{c} \gets |\hat{C}|$
		\State $c \gets |C|$
		\If{$\hat{c}$ it is not a prime power}
		\State \Return $C$ is of plain type
		\Else 
		\State Let $p$, $r$ be such that $\hat{c}=p^r$
		\If{$r<2$}
		\State \Return $C$ is of plain type \Comment{\textcolor{blue}{By Propositions \ref{propC_y} and \ref{CarachetisationN-classes_1}}}
		\EndIf 
		\State $s\gets 0$ \Comment{\textcolor{blue}{From now on $r\geq2$}}
		\State $FIND \gets 0$
		\While{$s\leq r-2$ and $FIND=0$}
		\If{$c=p^r-p^s$}
		\State $FIND \gets 1$
		\Else 
		\State $s \gets s+1$
		\EndIf
		\EndWhile
		\If{$FIND = 0$}
		\State \Return $C$ is of plain type \Comment{\textcolor{blue}{By Propositions  \ref{propC_y} and \ref{CarachetisationN-classes_1}}}
		\EndIf
		\If{$s\ne 0$}
		\State \Return $C$ is of compound type \Comment{\textcolor{blue}{By Proposition \ref{CarachetisationN-classes_2}}}
		\EndIf 
		
		\Comment{\textcolor{blue}{If we reach this point, then we have $|\hat{C}|=p^r$ and $\hat{C}=C \cupdot \{1\}$. Hence $C$ is critical}}
		\State Take a $y \in C$
		\ForAll{$x \in G\setminus \hat{C}$}
		\State $|X| \gets  |[x]_{\mathtt{N}}|$
		\If{$|X|\leq c$ and $\{x,y\}\in E$}
		\State \Return $C$ is of plain type \Comment{\textcolor{blue}{By Proposition \ref{LemmaCriticalClassMio}}}
		\EndIf
		\EndFor
		\State \Return $C$ is of compound type \Comment{\textcolor{blue}{By Proposition \ref{LemmaCriticalClassMio}}}
		\EndIf
		
	\end{algorithmic}
\end{algorithm}

\begin{algorithm}
	\caption{CompoundClassPartitioningAlgorithm}\label{CompoundClassPartitioningAlgorithm}
	\begin{algorithmic}[1]
		\Require $C$ a class of compound type, $(p,r,s)$ the parameters of $C$
		\Ensure The partition of $C$ in $\diamond$-classes, up to graph isomorphism
		\State $C=\{a_j \, : \, j\in [|C|]\}$ \Comment{\textcolor{blue}{We name the elements of $C$ as shown}}
		\State $m\gets 0$
		\For{$i = s+1$ to $r$}
		\State $n\gets \phi(p^i)$
		\State $X_i(C)\gets \{a_{m+1}, ..., a_{m+n}\}$
		\State $m\gets  m+n$
		\EndFor
		\State $\mathcal{K}_{C} \gets \{X_i(C) \, : \, s+1 \leq i \leq r \}$
		\State \Return $\mathcal{K}_{C}$
		
	\end{algorithmic}
\end{algorithm}

\begin{algorithm}
	\caption{ArcsDirectionsChoiceAlgorithm}\label{ArcsDirectionsChoiceAlgorithm}
	\begin{algorithmic}[1]
		\Require The set $V$ of all vertices of a power graph and the partition $\mathcal{K}_{\diamond}$ of $V\setminus \{1\}$ in $\diamond$-classes
		\Ensure The directed power graph
		\State $A_{\vec{\Gamma}} \gets \varnothing$ 
		\State $K \gets \varnothing$ \Comment{\textcolor{blue}{It is needed to avoid useless checks}} 
		
		\ForAll{$X \in \mathcal{K}_{\diamond}$} \Comment{\textcolor{blue}{We now give directions to the edges between elements in $X$}}
		\State $H\gets \varnothing$
		\ForAll{$x \in H$} 
		\State $H\gets H\cup \{x\}$
		\ForAll{$y \in X\setminus H$}
		\State $A_{\vec{\Gamma}}\gets  A_{\vec{\Gamma}} \cup \{(x,y),(y,x)\}$
		\EndFor
		\EndFor
		\State $K\gets K \cup \{X\}$
		\State $n_X \gets |X|$
		\ForAll{$Y \in \mathcal{K}_{\diamond}\setminus K$} 
		
		\Comment{\textcolor{blue}{Here we take care of the edges between the vertices in $X$ and the vertices in the other $\diamond$-classes different from the star class}}
		\If{$X$ is joined to $Y$}
		\State $n_Y \gets |Y|$
		\If{$n_X > n_Y$}
		\State $A_{\vec{\Gamma}}\gets  A_{\vec{\Gamma}} \cup \{ (x,y) | x\in X \mbox{ and } y \in Y\}$
		\ElsIf{$n_Y > n_X$}
		\State $A_{\vec{\Gamma}}\gets A_{\vec{\Gamma}} \cup \{ (y,x) | x\in X \mbox{ and } y \in Y\}$
		\ElsIf{$n_X=n_Y$}
		\If{$X$ is joined with an involution} 
		
		\Comment{\textcolor{blue}{Note that we recognise the involutions as the vertices in $\diamond$-classes distinct from $\mathcal{S}$ with cardinality equals to $1$}}
		\State $A_{\vec{\Gamma}} \gets A_{\vec{\Gamma}} \cup \{ (x,y) | x\in X \mbox{ and } y \in Y\}$
		\Else
		\State $A_{\vec{\Gamma}}\gets A_{\vec{\Gamma}} \cup \{ (y,x) | x\in X \mbox{ and } y \in Y\}$
		\EndIf
		\EndIf
		\EndIf
		\EndFor
		\EndFor
		\ForAll{$x\in V\setminus \{1\}$} \Comment{\textcolor{blue}{Here take care of all the edges with the identity as an end vertex}}
		\State $A_{\vec{\Gamma}}\gets A_{\vec{\Gamma}} \cup \{(x,1)\}$
		\EndFor
		\State
		\Return $\vec{\Gamma}= (V, A_{\vec{\Gamma}})$
	\end{algorithmic}
\end{algorithm}

	\begin{algorithm} \caption{Reconstruction Algorithm of Directed Power Graphs}

	\label{ReconstructionOfDirectedPG}
	\begin{algorithmic}[1]
		
	\Require  $\Gamma=(V,E)$ a power graph
		\Ensure $\vec{\Gamma}=(V,A)$ the directed power graph
		\State $s \gets |\mathcal{S}|$
		\If{$s>1$}
		\State $n \gets |V|$
		\If{$s=n $}
		\State \Return $\vec{\Gamma}$ as the directed power graph of the cyclic group of order $n$
		\EndIf
		\If{$s= 1+ \phi(n)$}
		\State \Return $\vec{\Gamma}$ be the directed power graph of the cyclic group of order $n$
		\EndIf
		\If{$s=2$}
		\State \Return $\vec{\Gamma}$ as the directed power graph of the generalised quaternion group of order $n$
		\EndIf
		\EndIf \Comment{\textcolor{blue}{From now on $s=1$}}
		\State $\mathcal{K}_{\diamond} \gets \varnothing$
		\State $U \gets \{1\}$
		\While{$U \ne V$}
		\State Take a $x \in V\setminus U$
		\State $C \gets [x]_{\mathtt{N}}$
		\State Type $\gets $ ClassTypeDistinctAlgorithm($\Gamma, C$) \Comment{\textcolor{blue}{Algorithm \ref{ClassTypeDistinctionAlgorithm}}}
		\If{Type = plain}
		\State $\mathcal{K}_{\diamond}\gets \mathcal{K}_{\diamond}\cup C$
		\State $U\gets \bigcup \{w \, |\, w \in \mathcal{K}_{\diamond}\} \cup \{1\}$
		\Else
		\State $(p,r,s) \gets$ parameters of $C$
		\State $\mathcal{K}_{C} \gets $ CompoundClassPartitioningAlgorithm($C, (p,r,s)$) \Comment{\textcolor{blue}{Algorithm \ref{CompoundClassPartitioningAlgorithm}}}
		\State $\mathcal{K}_{\diamond} \gets \mathcal{K}_{\diamond} \cup \mathcal{K}_{C}$
		\State $U \gets \bigcup \{w \, |\, w \in \mathcal{K}_{\diamond}\} \cup \{1\}$
		\EndIf
		\EndWhile
		\State $\vec{\Gamma} \gets $ ArcsDirectionsChoiceAlgorithm($V, \mathcal{K}_{\diamond}$) \Comment{\textcolor{blue}{Algorithm \ref{ArcsDirectionsChoiceAlgorithm}}}
		\State \Return $\vec{\Gamma}$
	\end{algorithmic}
\end{algorithm}

\end{document}